\documentclass[reqno, xcolor=pdftex,letterpaper]{amsart}

\usepackage[english]{babel}
\usepackage{amsmath,amsthm,amssymb,amsfonts}
\usepackage{mathrsfs}
\usepackage[utf8x]{inputenc}
\usepackage[normalem]{ulem}
\usepackage{cancel}

\usepackage{enumitem}
\setlist[enumerate]{leftmargin=*}

\usepackage{hyperref}
\usepackage[utf8x]{inputenc}

\newcommand{\e}{\mathrm{e}}

\newcommand{\diff}{\textrm{d}} 
\usepackage[non-compressed-cites,initials,nobysame]{amsrefs}

\def\bbX{{\mathbb X}}

\renewcommand{\approx}{ \asymp}

\DeclareMathOperator{\supp}{supp}

\DeclareFontFamily{U}{mathx}{\hyphenchar\font45}
\DeclareFontShape{U}{mathx}{m}{n}{
	<5> <6> <7> <8> <9> <10>
	<10.95> <12> <14.4> <17.28> <20.74> <24.88>
	mathx10
}{}
\DeclareSymbolFont{mathx}{U}{mathx}{m}{n}
\DeclareFontSubstitution{U}{mathx}{m}{n}
\DeclareMathAccent{\widecheck}{0}{mathx}{"71}
\DeclareMathAccent{\wideparen}{0}{mathx}{"75}

\makeatletter
\newcommand{\leqnomode}{\tagsleft@true}
\newcommand{\reqnomode}{\tagsleft@false}
\makeatother

\numberwithin{equation}{section}

\usepackage[usenames,dvipsnames]{color}

\newcommand{\dd}{\mathrm{d}}

\newcommand{\R}{\mathbb{R}}
\newcommand{\N}{\mathbb{N}}

\newcommand{\Ls}{\mathcal{L}}

\theoremstyle{theorem}
\newtheorem{theorem}{\sc \textbf{Theorem}}[section]  
\newtheorem{proposition}[theorem]{\sc \textbf{Proposition}}   
\newtheorem{corollary}[theorem]{\sc \textbf{Corollary}}        
\newtheorem{lemma}[theorem]{\sc \textbf{Lemma}}

\renewcommand{\approx}{ \asymp}

\theoremstyle{remark}
\newtheorem{definition}[theorem]{\sc \textbf{Definition}}

\newtheorem{remark}[theorem]{\sc \textbf{Remark}}

\usepackage[T1]{fontenc}
\DeclareFontFamily{T1}{calligra}{}
\DeclareFontShape{T1}{calligra}{m}{n}{<->s*[1.44]callig15}{}
\DeclareMathAlphabet\mathcalligra   {T1}{calligra} {m} {n}
\DeclareMathAlphabet\mathzapf       {T1}{pzc} {mb} {it}
\DeclareMathAlphabet\mathchorus     {T1}{qzc} {m} {n}
\DeclareMathAlphabet\mathrsfso      {U}{rsfso}{m}{n}

\makeatletter
\newcommand{\myitem}[1]{%
	\item[#1]\protected@edef\@currentlabel{#1}%
}
\makeatother

\begin{document}

	\title[Pointwise convergence to initial data on symmetric spaces]{Pointwise convergence to initial data for some evolution equations on symmetric spaces} 
	
	\author[T.\ Bruno]{Tommaso Bruno}
	\address{Dipartimento di Matematica, Universit\`a degli Studi di Genova\\ Via Dodecaneso 35, 16146 Genova, Italy}
	\email{brunot@dima.unige.it}
	
	\author[E.\ Papageorgiou]{Effie Papageorgiou}
	\address{Institut f{\"u}r Mathematik, Universit\"at Paderborn, Warburger Str. 100, D-33098
		Paderborn, Germany}
	\email{papageoeffie@gmail.com}

	\keywords{Symmetric spaces; heat equation; Poisson equation; extension problem; vector valued weak type $(1,1)$; local maximal function; pointwise convergence}
	\thanks{{\em Math Subject Classification} 22E30, 35B40, 26A33, 58J47, 35C15.}
	\thanks{The first-named author was partially supported by the 2022 INdAM--GNAMPA grant {\em Generalized Laplacians on continuous and discrete structures} (CUP\_E55F22000270001). The second-named author is funded by the Deutsche Forschungsgemeinschaft (DFG, German Research Foundation)--SFB-Gesch{\"a}ftszeichen --Projektnummer SFB-TRR 358/1 2023 --491392403.  This work was also partially supported by the Hellenic Foundation for Research and Innovation, Project HFRI-FM17-1733.} 
	
	\begin{abstract}
		Let $\Ls$ be either the Laplace--Beltrami operator, its shift without spectral gap, or the distinguished Laplacian on a symmetric space of noncompact type $\mathbb{X}$ of arbitrary rank. We consider the heat equation, the fractional heat equation, and the Caffarelli--Silvestre extension problem associated with $\Ls$, and in each of these cases we characterize the weights $v$ on $\mathbb{X}$ for which the solution converges pointwise a.e.\ to the initial data when the latter is in $L^{p}(v)$, $1\leq p < \infty$. As a tool, we also establish vector-valued weak type $(1,1)$ and $L^{p}$ estimates ($1<p<\infty$) for the local Hardy--Littlewood maximal function on $\mathbb{X}$.
	\end{abstract}
	
	\maketitle

	\section{Introduction}
	In recent years, the problem of characterizing the measures $\nu$ for which the solution to a given Cauchy problem converges pointwise almost everywhere to the initial data when this is in $L^{p}(\nu)$ has been attracting a great deal of attention. The heat equation, the Poisson equation, and more generally the fractional heat equation, among others, have been studied for several differential operators in a variety of settings, see e.g.~\cite{HTV, GHSTV, A-FST, C, A-RBB}. In this paper we consider this question on a symmetric space of the noncompact type of arbitrary rank, for three problems associated either to the Laplace--Beltrami operator, to its shift without spectral gap, or the distinguished Laplacian: the heat equation~\cites{AJ99, AO}, the fractional heat equation~\cite{G61,GS04}, and the Caffarelli--Silvestre extension problem~\cite{BanEtAl, BP2022}. 
	
	To be more precise, let $\mathbb{X} = G/K$ be a symmetric space of noncompact type and $\Delta$ be its positive Laplace--Beltrami operator whose spectral gap equals $\| \rho\|^{2}$ (more details about the notation will be discussed later). We consider the non-negative operator
	\[
	\mathcal{L}_{\zeta}=\Delta-\|\rho\|^2+\zeta^2, \qquad \zeta \in [0,\|\rho\|^{2}],
	\]
	and the associated Cauchy problems on $\mathbb{X}$, for $t>0$, $\alpha \in (0,2)$ and $\sigma \in(0,1)$,
	\begin{equation}\label{Cauchyintro}
		\begin{cases}	
			\partial_t u + \mathcal{\Ls}_{\zeta}u=0\\
			u( \, \cdot \, ,0) = f,
		\end{cases} \quad 
		\begin{cases}
			\partial_t u + \Ls_{\zeta}^{\alpha/2}u = 0\\
			u(\, \cdot \, ,0)=f,
		\end{cases} \: \:
		\begin{cases}
			\mathcal{L}_{\zeta} u - \frac{(1-2\sigma)}{t} \partial_{t} u -  \partial^{2}_{tt} u=0 \\
			u( \, \cdot \, ,0)=f.
		\end{cases}
	\end{equation}
	It is well known that the solution to each of these problems can be expressed by means of a positive bi-$K$-invariant right-convolution kernel $\psi_{t}$ on $G$, namely -- provided $f$ is sufficiently regular,
	\[
	u(x ,t) = f\ast \psi_t(x)=\int_{G}f(y)\psi_t(y^{-1}x)\,\diff y, \qquad x\in G,
	\]
	where $\dd y$ is a Haar measure on $G$. In particular, since $u = \e^{-t\Ls_{\zeta}} f$ for the heat equation and $u = \e^{-t\Ls_{\zeta}^{\alpha/2}} f$ for the fractional heat equation, $\psi_{t}$ is respectively the heat kernel and the so-called fractional heat kernel, which encompasses the Poisson kernel as this corresponds to $\alpha = 1$, for $\Ls_{\zeta}$. The case of the Caffarelli--Silvestre extension problem is more delicate, and will be discussed in due course.
	
	Denote now by $\mu$ the $G$-invariant measure on $\mathbb{X}$ induced by $\dd y$, by $L^{p}$ the standard Lebesgue space on $\mathbb{X}$ with respect to $\mu$, and by $L^{p}(v)$, for a positive and locally integrable function $v$ which we call \emph{weight}, the Lebesgue space with respect to the measure with density $v$ with respect to $\mu$. Then, we look for those weights $v$ such that $u(\, \cdot \, , t) \to f$ for $t \to 0^{+}$ $\mu$-a.e.\ in $\mathbb{X}$ whenever $f\in L^{p}(v)$. It is classical that these properties hold for $f\in L^{p}$, $1\leq p<\infty$, as they are related to the $L^{p}$ boundedness (or weak type $(1,1)$ boundedness if $p=1$) of the corresponding maximal operator
	\begin{equation}\label{maxintro}
		f\mapsto \sup_{t>0} | f* \psi_{t}|,
	\end{equation}
	or even to its "local" version, where the supremum is taken only for small times. This amounts to considering $v=1$. For general weights $v$ and $f\in L^{p}(v)$, a particular instance of our results reads as follows.
	\begin{theorem}\label{teointro}
		Suppose $1\leq p <\infty$ and let $v$ be a weight. If $u(\, \cdot \, , t) =f*\psi_{t}$ is the solution to any of the Cauchy problems~\eqref{Cauchyintro}, then $\lim_{t\to 0^{+}} u(\, \cdot \, ,t) =f$  $\mu$-a.e.\ for all $f\in L^{p}(v)$ if and only if there exists $t_0>0$ such that $\psi_{t_{0}}v^{-\frac{1}{p}} \in L^{p'}$.
	\end{theorem}
	The conditions above have a number of other equivalent characterizations which highlight the role of the maximal operator~\eqref{maxintro}, or rather its local version, but we refer to Theorem~\ref{teo:main} below for a complete statement. Let us say that such results are the symmetric space analogues for $\Ls_{\zeta}$ of those known in other settings for certain differential operators including the Laplacian~\cite{HTV}, the Hermite and the Ornstein--Uhlenbeck operators~\cite{GHSTV, A-FST}, the Bessel operator~\cite{C}, and the combinatorial Laplacian on a homogeneous tree~\cite{A-RBB}. To the best of our knowledge, this is the first time that such problems are considered in a continuous (locally Euclidean) setting of exponential growth. And though the general scheme of the proof is close to that of the seminal paper~\cite{HTV} as in the aforementioned~\cite{GHSTV, A-FST, C, A-RBB}, the pointwise behavior of local maximal operators and the geometry at infinity mix here in a novel way. It sounds natural to speculate to what extent the results of this paper can be extended to more general manifolds; but we leave this question to future work.
	
	Along the way to proving Theorem~\ref{teointro}, we also establish vector-valued $L^{p}$ ($1<p<\infty$) and weak type $(1,1)$ estimates  for the local Hardy--Littlewood maximal function on $\mathbb{X}$; see Proposition~\ref{vvw11}. This result seems the first of its kind in a nondoubling setting, and as such we believe that it has an independent interest.
	
	\smallskip
	
	In a similar manner, we also obtain a result in the exact same spirit of Theorem~\ref{teointro} (and of Theorem~\ref{teo:main} and Proposition~\ref{vvw11}) for the Cauchy problems~\eqref{Cauchyintro} associated with the distinguished Laplacian
	\[ 
	\tilde{\Delta}= \tilde{\delta}^{1/2} \, \Ls_{0} \, \tilde{\delta}^{-1/2}
	\]
	on the solvable group $S=N(\exp{\mathfrak{a}})$ which appears in the Iwasawa decomposition $N(\exp{\mathfrak{a}})K$ of $G$, where $\tilde{\delta}$ is the modular function of $S$. 
	
	\bigskip
	
	We finally summarize  the structure of the paper. The following Section~\ref{sec:SS} is devoted to describing the setting of the paper and to recalling the basic features of the analysis on symmetric spaces. In Section~\ref{Sec:Cauchy} we discuss the three Cauchy problems~\eqref{Cauchyintro} for $\Ls_{\zeta}$ and the kernels associated to their solutions, and obtain (in some cases recall) pointwise upper and lower bounds for these. In Section~\ref{Sec:maximal} we prove vector-valued inequalities for the local Hardy-Littlewood maximal function which will be needed at a later stage. In Section~\ref{sec:class} we introduce a suitable class of kernels to which the three aforementioned kernels belong and we discuss the pointwise convergence results for that class in Section~\ref{sec:pc}. In the final Section~\ref{sec:distinguished} we discuss analogous results for the Cauchy problems associated with the distinguished Laplacian $\tilde{\Delta}$.
	
	\section{Analysis on symmetric spaces}\label{sec:SS}
	
	Let $G$ be a semi-simple, connected, noncompact  Lie group with finite center, and $K$ be a maximal compact subgroup of $G$. The homogeneous space $\mathbb{X}=G/K$ is a Riemannian symmetric space of noncompact type, whose generic element we denote by $xK$ for $x\in G$. We identify functions on the symmetric space $\mathbb{X}$ with right-$K$-invariant functions on $G$, and left-$K$-invariant functions on $\mathbb{X}$ with bi-$K$-invariant functions on $G$ in the usual way.

	Let $\mathfrak{g}$ be the Lie algebra of $G$ and $\mathfrak{g}=\mathfrak{k}\oplus\mathfrak{p}$ be its Cartan decomposition. The Killing form of $\mathfrak{g}$ induces 
	a $K$-invariant inner product $\langle\,\cdot \,,\,\cdot \,\rangle$ on $\mathfrak{p}$, 
	hence a $G$-invariant Riemannian metric on $\mathbb{X}$. We denote by $d$ the Riemannian distance on $\mathbb{X}$.
	
	Fix a maximal abelian subspace $\mathfrak{a}$ in $\mathfrak{p}$. We identify $\mathfrak{a}$ with its dual $\mathfrak{a}^{*}$ 
	by means of the inner product inherited from $\mathfrak{p}$. We denote by $\Sigma \subset\mathfrak{a}$ the root system of $(\mathfrak{g},\mathfrak{a})$ and by $W$ the associated Weyl group. Choose a positive Weyl chamber $\mathfrak{a}^{+}\subset\mathfrak{a}$, and let $\Sigma^{+}$ and $\Sigma_{r}^{+}$ be respectively the corresponding set of positive roots and positive reduced (i.e., indivisible) roots. We denote by $n$ the dimension of $\mathbb{X}$, by $\nu$ be its dimension at infinity, and by $\ell$ its rank, i.e.\ the dimension of $\mathfrak{a}$. In particular
	\[
	n = \ell+\sum\nolimits_{\alpha \in \Sigma^{+}}\,m_{\alpha}, \qquad \qquad \nu = \ell+2|\Sigma_{r}^{+}|,
	\]
	where $m_{\alpha}$ denotes the dimension of the positive root subspace
	\[
	\mathfrak{g}_{\alpha}= \{X\in\mathfrak{g}\,|\,[H,X]=\langle{\alpha,H}\rangle{X}, \quad \forall\,H\in\mathfrak{a}\}.
	\]
	
	Let $\mathfrak{n}$ be the nilpotent Lie subalgebra 
	of $\mathfrak{g}$ associated with $\Sigma^{+}$ 
	and let $N = \exp \mathfrak{n}$ be the corresponding subgroup of $G$. The Iwasawa and Cartan decompositions of $G$ are respectively
	\[
	G=N (\exp\mathfrak{a}) K , \qquad 		G = K (\exp\overline{\mathfrak{a}^{+}}) K.
	\]
	Denote by $A(x)\in\mathfrak{a}$ and $x^{+}\in\overline{\mathfrak{a}^{+}}$
	the associated components of $x\in{G}$ in these two decompositions respectively, and by $|x|=\|x^{+}\|$. Viewed on $\mathbb{X}=G/K$, $|\cdot|$ is the distance to the origin $o=eK$, namely
	\[
	|x| = d(xK, eK), \qquad |y^{-1}x| = d(xK, yK)
	\]
	(the second equality by the $G$-invariance of the metric on $\mathbb{X}$). For all $x,y\in{G}$, we have
	\begin{align}\label{S2 Distance}
		\|A(x)\|\,\le\,|x| 
	\end{align}
	and
	\begin{align}\label{S2 flat}
		\|x^{+}-y^{+}\|\leq d(xK,yK), \quad \|(yx)^{+}-y^{+}\|, \;	\|(xy)^{+}-y^{+}\|\leq d(xK,eK), 
	\end{align}
	see for instance~\cite[Lemma 2.1.2]{AJ99} or~\cite[Lemma 2.1]{MMV}.
	
	We denote by $\dd x$ the Haar measure on $G$, and by $\dd k$ the normalized Haar measure of $K$. The Haar measure of $G$ induces a $G$-invariant measure $\mu$ on $\mathbb{X}$ for which, for good enough $f$,
	\[
	\int_{\mathbb{X}} f \, \dd \mu = \int_{G} f(x) \, \dd x.
	\]
	In the Cartan decomposition, in particular, the integral over $G$ takes the form
	\begin{align*}
		\int_{G}f(x)\, \dd x =\,\int_{K}\int_{\mathfrak{a}^{+}}\int_{K}f(k_{1}(\exp x^{+})k_{2})\, \delta(x^{+})\,\diff{k_1}\,\diff{x^{+}}\,  \diff{k_2},
	\end{align*}
	where  $\dd x^{+}$ is a suitable multiple of the Lebesgue measure on $\overline{\mathfrak{a}^{+}}$, and
	\[
	\delta(x^{+})\,
	=\,\prod_{\alpha\in\Sigma^{+}}\,
	(\sinh\langle{\alpha,x^{+}}\rangle)^{m_{\alpha}}\,
	\asymp\,
	\prod_{\alpha\in\Sigma^{+}}
	\bigg( 
	\frac{\langle\alpha,x^{+}\rangle}
	{1+\langle\alpha,x^{+}\rangle}
	\bigg)^{m_{\alpha}}\,
	\e^{2\langle\rho,x^{+}\rangle}
	\qquad x^{+}\in\overline{\mathfrak{a}^{+}},
	\]
	where  $\rho\in\mathfrak{a}^{+}$ is the half-sum of all positive roots  $\alpha \in \Sigma^{+}$ counted with their multiplicities $m_{\alpha}$, i.e.,
	\begin{align*}
		\rho =
		\frac{1}{2}\,\sum\nolimits_{\alpha\in\Sigma^{+}} \,m_{\alpha}\,\alpha.
	\end{align*}	
	For $r>0$, finally, define the balls
	\begin{align*}
		B(o, r)&=\{xK\in \mathbb{X}: \; d(xK, eK)\leq r\},\\
		B(e, r)&=\{ x \in G: \; |x|\leq r\}.
	\end{align*}
	We recall (see e.g.~\cite{Kni97}) that for all $R>0$
	\begin{align}\label{volgrowth}
		\mu(B(xK,r))\approx_{R}
		\begin{cases}
			r^n \; &\text{if } 0<r\leq R,
			\\[7pt]
			r^{\frac{\ell-1}{2}}\, \e^{2\|\rho\|r} \; &\text{if } r>R,
		\end{cases}
	\end{align}
	whence the metric measure space $(\mathbb{X}, d, \mu)$ is locally but not globally doubling. Here and all throughout, we write $f \approx g$ for two positive functions $f$ and $g$ whenever there exists $C > 0$ (depending on $\mathbb{X}$ and other circumstantial parameters) such that $C^{-1}g \leq f \leq C g$. Analogously, we shall write $f\lesssim g$ if there exists such a $C$ such that $f\leq C g$. In order to stress the dependence of $C$ on a specific parameter, say $\kappa$, we write $f \approx_{\kappa} g$ or $f\lesssim_{\kappa} g$ as in~\eqref{volgrowth}.

	For $1\leq p\leq \infty$, we shall denote by $L^{p}$ the standard Lebesgue spaces with respect to $\mu$, and by $\| \cdot \|_{p}$ their norm. If $v$ is a weight, $L^{p}(v)$ will stand for the Lebesgue space with respect to the measure with density $v$ with respect to $\mu$, i.e.\ $\dd v= v\, \dd \mu$.  Its norm will be $\| \cdot \|_{L^{p}(v)}$. We shall denote by $C_{c}$, and $C_{c}^{\infty}$ respectively, the sets of continuous compactly supported and smooth and compactly supported functions on $\mathbb{X}$.

	
	\subsection{Spherical analysis and $\Ls_{\zeta}$}
	Let $\mathcal{S}(K \backslash{G}/K)$ be the Schwartz space of bi-$K$-invariant
	functions on $G$. The spherical Fourier transform (or Harish--Chandra transform)
	$\mathcal{H}$ is defined to be
	\begin{equation}\label{SFtransform}
		\mathcal{H}f(\lambda)\,
		=\,\int_{G}\varphi_{-\lambda}(x)\,f(x) \, \diff{x}
		\qquad \lambda\in\mathfrak{a},\quad f\in\mathcal{S}(K\backslash{G/K}),
	\end{equation}
	where $\varphi_{\lambda}\in C^{\infty}(K\backslash{G/K})$ is the spherical function of index $\lambda$.
	If $\mathcal{S}(\mathfrak{a})^{W}$ stands for the subspace 
	of $W$-invariant functions in the Schwartz space $\mathcal{S}(\mathfrak{a})$, then $\mathcal{H}$ is an isomorphism between $\mathcal{S}(K\backslash{G/K})$ 
	and $\mathcal{S}(\mathfrak{a})^{W}$. The inverse spherical Fourier transform 
	is given by
	\[
	f(x)= C_0 \int_{\mathfrak{a}}\,\mathcal{H}f(\lambda) \,\varphi_{\lambda}(x)\,|\mathbf{c(\lambda)}|^{-2}\,\diff{\lambda}\qquad x\in{G}, \quad \mathcal{H}f\in\mathcal{S}(\mathfrak{a})^{W},
	\]
	where $C_0$ is a suitable constant which depends only on $\mathbb{X}$, and $|\mathbf{c(\lambda)}|^{-2}$ is the  so-called Plancherel density which satisfies, for some $m_{0}>0$,
	\begin{equation}\label{HCest}
		|\textbf{c}(\lambda)|^{-2}\lesssim (1+\|\lambda\|)^{m_{0}}, \qquad  \lambda \in \mathfrak{a}.
	\end{equation}
	
	For $\lambda\in\mathfrak{a}$, the spherical function $\varphi_{\lambda}$ is a smooth bi-$K$-invariant eigenfunction of all 
	$G$-invariant differential operators on $\mathbb{X}$, in particular of the (positive)
	Laplace--Beltrami operator $\Delta$, i.e.\
	\[
	\Delta\varphi_{\lambda} = (\|\lambda\|^{2}+\|\rho\|^2)\,\varphi_{\lambda}.
	\]
	We recall moreover that for all $\lambda\in\mathfrak{a}$ the function $\varphi_{\lambda}$ is controlled by the ground spherical
	function $\varphi_{0}$, and that this satisfies the global estimate
	\begin{equation}\label{phi_0 global}
	\varphi_{0}(x) \asymp \bigg\{ \prod_{\alpha\in\Sigma_{r}^{+}}  (1+\langle\alpha,x^{+}\rangle)\bigg\}\, \e^{-\langle\rho, x^{+}\rangle}, \qquad x\in G,
	\end{equation}
	(cf.~\cite[Proposition 2.2.12]{AJ99}). Observe that this implies that 
	\begin{equation}\label{phi_0 upper}
		\e^{-\|\rho\||x|}\lesssim\varphi_{0}(x)\lesssim  (1+|x|)^{|\Sigma_r^{+}|}\e^{-\rho_{\min}|x|}, \qquad x\in G,
	\end{equation}
	where 
	\[
	\rho_{\min} =\min_{x^{+}\in\mathfrak{a}^{+}, \; \|x^{+}\|=1}  \langle \rho, x^{+} \rangle \in(0, \|\rho\|].
	\]
	We will also need that
	\begin{align}\label{phi split}
		\varphi_{0}(x^{-1}y)\,
		=\,\int_{K}\e^{\langle{\rho,A(kx})\rangle}\,
		\e^{\langle{\rho,A(ky})\rangle}\,\diff{k},
	\end{align}
	see e.g.~\cite[Ch.~III, Theorem~1.1]{Hel1994}. From now on, we shall denote by $\Ls_{\zeta}$ the non-negative operator
	\[
	\mathcal{L}_{\zeta}=\Delta-\|\rho\|^2+\zeta^2, \qquad \zeta \in [0,\|\rho\|],
	\]
	which is a shift of the Laplace--Beltrami operator. Notice that its $L^{2}$ spectrum is $\sigma(\Ls_{\zeta}) = [\zeta^{2},\infty)$, and in particular $\Ls_{\|\rho\|} = \Delta$. If not otherwise stated, $\zeta$ will always lie in $[0,\|\rho\|]$ in the following.


	\section{Cauchy problems for some evolution equations}\label{Sec:Cauchy}
	In this section we introduce and discuss the Cauchy problems~\eqref{Cauchyintro} in detail.
	\subsection{Heat equation and heat kernel} 
	We consider the Cauchy problem for the heat equation associated to $\Ls_{\zeta}$ 
	\begin{equation*}
		\begin{cases}
			\partial_t u + \mathcal{\Ls}_{\zeta}u=0,\qquad t>0,\\
			u(\, \cdot \, ,0) = f
		\end{cases}
	\end{equation*}
	and recall that its solution, for sufficiently regular $f$, is given by $u(\, \cdot \, ,t) = \e^{-t\Ls_{\zeta}}f$. If $H_{t} $ stands for the integral kernel of the operator $\e^{-t\Delta}$, then the function
	\[
	h_{t}(xK) = H_{t}(xK,eK), \qquad x\in G,
	\]
	is called the \emph{heat kernel} on $\mathbb{X}$, and it is a right convolution kernel in the sense that 
	\[
	\e^{-t\Delta} f = f * h_{t}.
	\]
	Following the identification of functions on $\mathbb{X}$ with right-$K$-invariant functions on $G$, we recall that $h_{t}$ is a positive bi-$K$-invariant function on $G$, hence determined by its restriction to the positive Weyl chamber. By the inversion formula of the spherical Fourier transform, one has
	\[
	h_{t}(x)=C_{0}\,\int_{\mathfrak{a}}\,\e^{-t(\|\lambda\|^{2}+\|\rho\|^{2})}\varphi_{\lambda}(x)\,|\mathbf{c(\lambda)}|^{-2}\,\diff{\lambda}, \qquad x\in G,\: t>0.
	\]
	It is known that $h_{t}$ is also symmetric, i.e. $h_t(x)=h_t(x^{-1})$. All the above holds in turn for the heat kernel $h_t^{\zeta}$ of $\Ls_{\zeta}$ as
	\begin{equation}\label{htzetaht}
		h_t^{\zeta}=\e^{(\|\rho\|^2-\zeta^2)t} \, h_t.
	\end{equation}
	Notice that $h_{t} = h_{t}^{\|\rho\|}$ and that
	\begin{equation}\label{zetageqrho}
		h_{t}^{\zeta}(x) \geq h_{t}^{\|\rho\|}(x),\qquad x\in G, \quad \zeta \in [0,\|\rho\|].
	\end{equation}
	By the bilateral estimates for $h_{t}$ shown in~\cites{AJ99, AO} and~\eqref{htzetaht},
	\begin{align}\label{S2heatkernel}
		h_{t}^{\zeta}(x)\,
		\asymp \, t^{-\frac{n}{2}}\,
		\bigg\{
		\prod_{\alpha\in\Sigma_{r}^{+}}
		(1+t+\langle{\alpha,x^{+}}\rangle)^{\frac{m_{\alpha}+m_{2\alpha}}{2}-1}
		\bigg\} \,\varphi_{0}(x)\,
		\e^{-\zeta^{2}t-\frac{|x|^{2}}{4t}}
	\end{align}
	for all $t>0$ and {$x\in G$}. For future convenience, we denote by $H_{t}^{\zeta}$ the integral kernel of the operator $\e^{-t\Ls_{\zeta}}$, so that $H_{t} = H_{t}^{\|\rho\|}$.

	\subsection{The Caffarelli--Silvestre extension problem} \label{subsection CS}
	The interest in the Caffarelli--Silvestre extension problem, which for $\sigma\in (0,1)$ reads
	\begin{align}\label{extension}
		\begin{cases}
			\mathcal{L}_{\zeta} u - \frac{(1-2\sigma)}{t} \partial_{t} u - \partial^{2}_{tt} u=0, \qquad  t>0, \\
			u(\, \cdot \, ,0)=f
		\end{cases}
	\end{align}
	originates from the study of the fractional powers $\mathcal{L}_{\zeta}^{\sigma}$, as we now illustrate.
	
	First, we note that for $t\in (0,1)$ and all $x\in G$, by an explicit computation using~\eqref{S2heatkernel} and~\eqref{phi_0 global} one has
	\[
	\|H_t(xK,\, \cdot \, )\|_2 \lesssim t^{-\frac{n}{4}}, \quad \|\partial_t H_t(xK, \, \cdot \, )\|_{2}\lesssim t^{-\frac{n}{4}-1}
	\]
	(for the time derivative, one may use the pointwise estimate~\cite[(3.1)]{AN92}). This together with the fact that
	\[
	H_t^{\zeta}=\e^{(\|\rho\|^2-\zeta^2)t}H_t\asymp H_t, \qquad t\in (0,1),
	\]
	and 
	\[
	\partial_t (\e^{(\|\rho\|^2-\zeta^2)t}H_t)=\e^{(\|\rho\|^2-\zeta^2)t}\big((\|\rho\|^2-\zeta^2)\,H_t+\partial_t H_t\big)
	\]
	implies that there is a constant $C>0$ such that for all $x\in G$ and $t\in (0,1)$
	\begin{equation}\label{kernel cond}
		\|H_t^{\zeta}(xK,  \, \cdot \,  )\|_{2}+\|\partial_t H_t^{\zeta}(xK, \, \cdot \,  )\|_{2} \leq C \, t^{-\frac{n}{4}-1}.
	\end{equation}
	By~\cite[Theorem 3.1 and \S 3.2]{BanEtAl}, for  $f\in \mathrm{Dom}(\mathcal{L}_{\zeta}^{\sigma})$, a solution to the Caffarelli--Silvestre extension problem~\eqref{extension} is given by
	\[
	T^{\sigma, \zeta}_tf:=u(\, \cdot \, , t)=\frac{t^{2\sigma}}{2^{2\sigma}\Gamma(\sigma)}\int_{0}^{+\infty}\e^{-\frac{t^2}{4s}}\, \e^{-s\mathcal{L}_{\zeta}}f\, \frac{\dd s}{s^{1+\sigma}},
	\]
	and the fractional Laplacian $\Ls_{\zeta}^{\sigma}$ on $\mathbb{X}$ can be recovered through
	\begin{equation*}\mathcal{L}_{\zeta}^{\sigma}f =-2^{2\sigma-1}\frac{\Gamma(\sigma)}{\Gamma(1-\sigma)}\lim_{t\rightarrow 0^{+}}t^{1-2\sigma}\frac{\partial u}{\partial t}(\, \cdot \,,t).
	\end{equation*}
	Owing to~\cite[Theorem 3.2]{BanEtAl} and~\eqref{kernel cond}, moreover, the operator $T_t^{\sigma, \zeta}$ admits a right convolution bi-$K$-invariant kernel
	\begin{equation}\label{kernelCS}
		Q_t^{\sigma, \zeta}(x)=\frac{t^{2\sigma}}{2^{2\sigma}\Gamma(\sigma)}\int_{0}^{+\infty} h_s^{\zeta}(x)\,\e^{-\frac{t^2}{4s}}\frac{\diff s}{s^{1+\sigma}}, \quad x\in G, \; t>0,
	\end{equation}
	which is positive and symmetric as so is $h_{s}^{\zeta}$. Since for $\sigma=1/2$ the problem~\eqref{extension} reduces to the Laplace equation for $\Ls_{\zeta}- \partial_{tt}^{2}$ and $(Q_t^{1/2, \zeta})$ is the Poisson kernel, we shall call $(Q_t^{\sigma, \zeta})$ the \emph{fractional Poisson kernel}.
	
	Upper and lower estimates for $Q_t^{\sigma, \|\rho\| }$ for $\sigma \in (0,1)$ were obtained in~\cite[Theorem 3.2]{BP2022} while for $Q_t^{1/2, 0}$ in~\cite{AJ99}. Since we shall need such estimates for all $\zeta \in [0,\|\rho\|]$, we prove them in the next proposition.
	\begin{proposition}\label{Q kernel estimates}
		(i) Suppose $0<\zeta \leq \|\rho\|$. Then for $\kappa >0$
		\begin{align}\label{Qzetapos}
			Q_t^{\sigma, \zeta}(x)\asymp_{\kappa,\zeta}	
			\begin{cases}
				t^{2\sigma}(t+|x|)^{-(n+2\sigma)} \; &\text{if } t+|x| \leq \kappa,
				\\[5pt]
				t^{2\sigma} (t+|x|)^{-\frac{\ell}{2}-|\Sigma_r^+|-\sigma-\frac{1}{2}}\, \varphi_0(x) \,\e^{-\zeta\sqrt{t^2+|x|^2}}  \; &\text{if } t+|x|\geq \kappa.
			\end{cases}
		\end{align}
		(ii) In the limit case $\zeta=0$, for $\kappa >0$ we have instead
		\begin{align}\label{Qzeta0}
			Q_t^{\sigma, 0}(x)\asymp_{\kappa}
			\begin{cases}
				t^{2\sigma}(t+|x|)^{-(n+2\sigma)} \; &\text{if } t+|x|\leq \kappa,
				\\[5pt]
				t^{2\sigma} (t+|x|)^{-\ell-2|\Sigma_r^+|-2\sigma}\, \varphi_0(x) \; &\text{if } t+|x| \geq \kappa.
			\end{cases}
		\end{align}
	\end{proposition}
	
	\begin{proof}
		We begin by assuming $t+|x|\leq \kappa$ and $\zeta \in [0,\|\rho\|]$.
		
		The lower bounds in~\eqref{Qzetapos} and~\eqref{Qzeta0} can be obtained by the lower estimates for $Q_t^{\sigma, \|\rho\|}$ obtained in~\cite[Theorem 3.2]{BP2022}. Indeed, as a consequence of the subordination formula~\eqref{kernelCS} and of~\eqref{zetageqrho}, one has
		\[
		Q_{t}^{\sigma, \zeta}(x) \geq Q_t^{\sigma, \|\rho\|}(x), \qquad x\in G. 
		\]
		As for the upper bounds, we split the subordination formula~\eqref{kernelCS} and get 
		\begin{align*}	
			Q_t^{\sigma, \zeta}(x)&\lesssim\ t^{2\sigma}\left( \int_{0}^{1}s^{-\frac{n}{2}}\,\e^{-\frac{t^2+|x|^2}{4s}}\frac{\diff s}{s^{1+\sigma}} +\int_{1}^{+\infty}\,s^{-\frac{\nu}{2}}\,\e^{-\zeta^2 s}\e^{-\frac{t^2}{4s}}\frac{\diff s}{s^{1+\sigma}}  \right),
		\end{align*}
		where since $|x|\leq \kappa$ we used for the first integral that 
		\[
		h_s^{\zeta}(x) \asymp s^{-\frac{n}{2}\e^{-\frac{|x|^2}{4s }}}, \qquad s\in(0,1),		\]
		while in the second integral we used that 
		\[
		\|h_s^{\zeta}\|_{\infty}\asymp s^{-\frac{\nu}{2}}\e^{-\zeta^2s}, \qquad s\geq 1.
		\]
		Both estimates used~\eqref{S2heatkernel}. Therefore, 
		\[
		Q_t^{\sigma, \zeta}(x)\lesssim t^{2\sigma} \big((t+|x|)^{-(n+2\sigma)} +1\big) \lesssim t^{2\sigma} (t+|x|)^{-(n+2\sigma)}.
		\]
		
		Assume now that $t+|x| \geq \kappa$. We separate the cases when $\zeta>0$ and $\zeta=0$.
		
		\smallskip
		
		Suppose $\zeta>0$ and fix a positive sufficiently large number $\eta>0$. Then, we split the integral in~\eqref{kernelCS} into three integrals, the first over the interval
		$(0, \eta^{-1}b)$, the second over $[\eta^{-1}b, \eta b]$ and the third over $[\eta b, +\infty)$, where $b=\sqrt{t^2+|x|^2}/(2\zeta)$. By proceeding as in the proof of~\cite[Theorem 4.3.1]{AJ99} or~\cite[Theorem 3.2]{BP2022}, one sees that the main contribution comes from the integral over $[\eta^{-1}b, \eta b]$, while the other two decay exponentially faster with respect to $t^2+|x|^2$. We omit the details.
		
		\smallskip
		
		Suppose now that $\zeta=0$. For the upper bound, using~\eqref{kernelCS} and~\eqref{S2heatkernel}, we have
		$$\frac{2^{2\sigma}\Gamma(\sigma)}{t^{2\sigma}}	Q_t^{\sigma, 0}(x)\,\varphi_{0}(x)^{-1} \asymp\int_{0}^{+\infty}s^{-\frac{n}{2}}\,
		\bigg\{
		\prod_{\alpha\in\Sigma_{r}^{+}}
		(1+s+\langle{\alpha,x^{+}}\rangle)^{\frac{m_{\alpha}+m_{2\alpha}}{2}-1}
		\bigg\}\e^{-\frac{t^2+|x|^2}{4s}}\frac{\diff s}{s^{1+\sigma}}$$
		and we further split the right hand side
		into two integrals over $(0, t^2+|x|^2]$ and $(t^2+|x|^2, +\infty)$. The claim follows in the first interval by estimating 
		\[
		\prod_{\alpha\in\Sigma_{r}^{+}}
		(1+s+\langle{\alpha,x^{+}}\rangle)^{\frac{m_{\alpha}+m_{2\alpha}}{2}-1}\lesssim (t^2+|x|^2)^{\frac{n-\ell}{2}-|\Sigma_r^{+}|}
		\]
		while in the second by \[
		\prod_{\alpha\in\Sigma_{r}^{+}}
		(1+s+\langle{\alpha,x^{+}}\rangle)^{\frac{m_{\alpha}+m_{2\alpha}}{2}-1}\lesssim s^{\frac{n-\ell}{2}-|\Sigma_r^{+}|}.
		\]
		The lower bound follows by estimating $ 1+s+\langle \alpha, x^+\rangle \geq  s$ and restricting the integral over $(t^2+|x|^2, +\infty)$. The details are left to the reader.
	\end{proof}

	\subsection{The fractional heat equation}
	We already observed that	the heat equation for $\sqrt{\Ls_{\zeta}}$, and consequently the Poisson semigroup, arise from the Caffarelli--Silvestre extension problem discussed above. However, they appear as well from the fractional heat equation, which we now illustrate.
	
	Suppose $\alpha \in (0,2)$ and let $\eta_{t}^{\alpha}$ be the inverse Laplace transform of the function $s \mapsto \exp\{-t s^{\alpha/2}\}$.
	The fractional Laplacian $\mathcal{L}_{\zeta}^{\alpha/2}$ is the infinitesimal generator of a standard isotropic $\alpha$-stable L{\'e}vy motion $X^{\alpha}_t$.
	This process is a L{\'e}vy process, which can be viewed as the long-time scaling limit of a random walk with power law jumps (\cite[Theorem 6.17]{MeSi2011}). Via subordination to the heat semigroup, we may write 
	\[
	\e^{-t\Ls_{\zeta}^{\alpha/2}}=\int_0^{+\infty}\eta_{t}^{\alpha}(s)\, \e^{-s\Ls_{\zeta}} \,\diff{s},
	\]
	see e.g.~\cite[(7), p.260]{Y}. Then, the solution to the initial value problem 
	\begin{equation*}
		\begin{cases}
			\partial_t u + \Ls_{\zeta}^{\alpha/2}u = 0, \qquad \; t>0,\\
			u(\, \cdot \,,0)=f
		\end{cases} 
	\end{equation*} 
	is given by the right convolution (for sufficiently regular $f$'s)
	\[
	u(x,t)=\e^{-t\Ls_{\zeta}^{\alpha/2}}f(x)=f\ast P_t^{\alpha, \zeta}(x)=\int_{G} P_t^{\alpha, \zeta}(y^{-1}x)f(y)\,\diff{y}, \qquad x\in G, \, t>0,
	\]
	where the kernels are given by
	\begin{equation}\label{kernelFH}
		P_t^{\alpha, \zeta}(x)=\int_0^{+\infty} h_s^{\zeta}(x)\,\eta_{t}^{\alpha}(s)\,\diff{s}, \qquad x\in G , \, t>0.
	\end{equation}
	This subordination formula implies that $P_t^{\alpha, \zeta}$, which we call \emph{fractional heat kernel}, is bi-$K$-invariant, positive and symmetric. It is worth mentioning that $(\e^{-t\Ls_{\zeta}^{\alpha/2}})_{t>0}$ is a $C_0$ semigroup  in $L^p$ for $1\leq p<\infty$, cf.~\cite{Y}. See also~\cite{G61, GS04}. 
	
	Precise estimates for $\eta_{t}^{\alpha}$ from above and below are known. In particular, see e.g.~\cite[(8)~and~(9)]{GS04},
	\begin{align} \label{etaUL}
		\eta_{t}^{\alpha}(s) \asymp
		\begin{cases}
			t^{\frac{1}{2-\alpha}}\, s^{-\frac{4-\alpha}{4-2\alpha}}\,\e^{-c_{\alpha} t^{\frac{2}{2-\alpha}} s^{-\frac{\alpha}{2-\alpha}}} &\mbox{if } s\leq t^{2/\alpha}, \\[5pt]
			t\, s^{-1-\frac{\alpha}{2}}
			&\mbox{if } s> t^{2/\alpha},
		\end{cases} \qquad t>0,
	\end{align}
	where $c_{\alpha}=\frac{2-\alpha}{2}\left( \frac{\alpha}{2}\right)^{\frac{\alpha}{2-\alpha}}$. From these, one also gets
	\begin{equation}\label{etaU}
		\eta_{t}^{\alpha}(s) \lesssim	t\, s^{-1-\frac{\alpha}{2}}, \qquad  t, s>0.
	\end{equation}

	Our aim is to prove the following.
	\begin{proposition}\label{P kernel estimates}
		(i) Suppose $0<\zeta  \leq \|\rho\|$. Then for $\kappa >0$
		\begin{align*}
			P_t^{\alpha, \zeta}(x)\asymp_{\kappa, \zeta}	
			\begin{cases}
				t\,(t^{1/\alpha}+|x|)^{-(n+\alpha)} & \quad \text{if } t+|x| \leq \kappa,
				\\[5pt]
				t\,(t+|x|)^{-\frac{\ell}{2}-|\Sigma_r^{+}|-\frac{\alpha}{2}-\frac{1}{2}}\,\varphi_0(x)\,\e^{-\zeta|x|} &\quad  \text{if } t+|x| \geq \kappa \text{ and } |x|\geq t^{2/\alpha}.
			\end{cases}
		\end{align*}
		(ii) In the limit case $\zeta=0$, for $\kappa >0$ we have instead
		\begin{align*}
			P_t^{\alpha, 0}(x)\asymp	
			\begin{cases}
				t\,(t^{1/\alpha}+|x|)^{-(n+\alpha)} \; & \quad \text{if } t+|x|\leq \kappa,
				\\[5pt]
				t\,(t^{1/\alpha}+|x|)^{-\ell-2|\Sigma_r^{+}|-\alpha}\,\varphi_0(x) &\quad \text{if } t+|x|\geq \kappa \text{ and } |x|\geq t^{1/\alpha}.
			\end{cases}
		\end{align*}
	\end{proposition}
	
	
	\begin{proof}
		Assume first that $t+|x| \leq \kappa$ and $\zeta \in [0,\|\rho\|]$. 
		
		For the lower bounds, it suffices to observe that for all $\zeta\in [0, \|\rho\|]$, by the subordination formula~\eqref{kernelFH} and~\eqref{zetageqrho}
		\[
		P_{t}^{\alpha, \zeta}(x) \geq P_{t}^{\alpha, \|\rho\|}(x) , \qquad  x\in G. 
		\]
		Then we can conclude using the lower bounds for $P_{t}^{\alpha, \|\rho\|}$ obtained in~\cite[Theorem 4.3]{GS04}.
		
		For the upper bound, observe that since both $t$ and $|x|$ are bounded
		\begin{equation}\label{small x heat}
			h_s^{\zeta}(x)\lesssim s^{-\frac{n}{2}}(1+s)^{\frac{n-\nu}{2}}\,\e^{-\frac{|x|^2}{4s}}, \qquad s>0.
		\end{equation}
		By the subordination formula~\eqref{kernelFH}, by~\eqref{small x heat} and~\eqref{etaU}, given $\kappa' =\kappa^{2/\alpha}$ we split
		\begin{align*}	
			P_t^{\alpha, \zeta}(x)&\lesssim\  \int_{0}^{\kappa'}\eta_t^{\alpha}(s)\,s^{-\frac{n}{2}}\,\e^{-\frac{|x|^2}{4s}}\diff s +t\int_{\kappa'}^{+\infty}\,s^{-1-\frac{\alpha}{2}}\,s^{-\frac{\nu}{2}}\diff s   =: I_1+I_2.
		\end{align*}
		As for $I_{2}$, it is easily seen that since $|x|+t\leq \kappa$
		\begin{align*}	
			I_2\lesssim\ t \lesssim t\,(t^{\frac{1}{\alpha}}+|x|)^{-(n+\alpha)}.
		\end{align*}
		To estimate $I_1$, we consider the cases when $|x|\geq t^{1/\alpha}$ and $|x|< t^{1/\alpha}$ separately.
		
		Suppose $|x|\geq t^{1/\alpha}$. Then
		\[
		|x|^2\asymp t^{\frac{2}{\alpha}}+|x|^2 \asymp (t^{\frac{1}{\alpha}}+|x|)^2,
		\]
		whence by~\eqref{small x heat} and~\eqref{etaU} we get for a suitable $c>0$
		\begin{align*}
			I_{1}&\lesssim t \int_{0}^{\kappa'}s^{-1-\frac{\alpha}{2}} \, s^{-\frac{n}{2}} \exp\bigg\{{-c\, \bigg(\frac{t^{1/\alpha}+|x|}{\sqrt{s}}\bigg)^2}\bigg\}\,\diff{s} \lesssim \frac{t}{(t^{1/\alpha}+|x|)^{n+\alpha}} .
		\end{align*}
		Suppose now that  $|x|< t^{1/\alpha}$. Then
		\begin{equation}\label{sim}
			\frac{t}{(t^{1/\alpha}+|x|)^{\alpha}}\asymp 1.
		\end{equation}
		We split $I_{1}$ in turn into the integrals over $(0, t^{2/\alpha})$ and $[t^{2/\alpha}, \kappa')$, which we denote by $I_{1,1}$ and $I_{1,2}$ respectively. By~\eqref{small x heat} and~\eqref{etaUL} we then have
		\begin{align*}
			I_{1,1}&\lesssim	t^{\frac{1}{2-\alpha}}	\int_{0}^{
				t^{2/\alpha}} s^{-\frac{4-\alpha}{4-2\alpha}}\,\e^{-c_{\alpha}\left(\frac{t^{2/\alpha}}{s}\right)^{\frac{\alpha}{2-\alpha}}}\;s^{-\frac{n}{2}}\;\e^{-\frac{|x|^2}{4s}}\;\diff{s} \\
			&	\lesssim t^{-n/\alpha} \int_{1}^{\infty} s^{\frac{3\alpha-4}{2(2-\alpha)}}\,\e^{-c_{\alpha} s^{\frac{\alpha}{2-\alpha}}}\;s^{\frac{n}{2}}\;\diff{s}  \\
			&\lesssim \frac{t}{( t^{1/\alpha}+|x|)^{n+\alpha}},
		\end{align*}
		after a change of variables and a use of~\eqref{sim}. Finally, we turn to $I_{1,2}$. Again by~\eqref{small x heat} and~\eqref{etaUL}
		\begin{align*}
			I_{1,2}&\lesssim	
			t	\int_{t^{2/\alpha}}^{\infty} s^{-1-\frac{\alpha}{2}}  \;s^{-\frac{n}{2}}\;\e^{-\frac{|x|^2}{4s}}\;\diff{s} \\
			&\lesssim	t^{-n/\alpha} \int_{0}^{1} s^{-1+\frac{\alpha}{2}+\frac{n}{2}} \;\diff{s} \lesssim t^{-n/\alpha}\\
			&\lesssim \frac{ t }{( t^{1/\alpha}+|x|)^{n+\alpha}}, 
		\end{align*}
		using~\eqref{sim} once more. The proof of the estimate when $t+|x| \leq \kappa$ is then complete.
		
		\smallskip
		
		Assume now that $t+|x| \geq \kappa$ and $|x|\geq t^{2/\alpha}$. We split the cases when $\zeta>0$ and $\zeta=0$.
		
		\smallskip
		
		Suppose $\zeta\in (0,\|\rho\|]$. Using the subordination formula~\eqref{kernelFH} and~\eqref{etaUL} we get 
		\begin{align*}
			P_t^{\alpha, \zeta}&\asymp t^{\frac{1}{2-\alpha}}\,
			\int_{0}^{t^{2/\alpha}} h_s^{\zeta}\,s^{-\frac{4-\alpha}{4-2\alpha}}\,\e^{-c_{\alpha} t^{\frac{2}{2-\alpha}} s^{-\frac{\alpha}{2-\alpha}}}\diff s  + t \int_{t^{2/\alpha}}^{+\infty} h_s^{\zeta}\,
			s^{-1-\frac{\alpha}{2}}\,\e^{-c_{\alpha} t^{\frac{2}{2-\alpha}} s^{-\frac{\alpha}{2-\alpha}}}\diff s.
		\end{align*}
		As in?\cite[p.\ 1075]{AJ99} or~\cite[p.\ 90]{GS04} one can show that the main contribution above comes from the exponential terms and their integration over the interval in which $s\asymp s_0(x, t)$, where $s_0$ minimizes the function
		$$f(s)=\frac{|x|^2}{4s}+\zeta^2s +c_{\alpha} t^{\frac{2}{2-\alpha}} s^{-\frac{\alpha}{2-\alpha}}.$$
		As $\zeta>0$ the result follows from a straightforward adaptation of the arguments in~\cite[\S 4--5]{GS04}, see in particular~\cite[Corollary 5.6]{GS04}. We omit the details.
		
		\smallskip
		
		Suppose now that $\zeta=0$. For the upper bound, we observe first that $|x|\gtrsim 1$, and then we use the subordination formula~\eqref{kernelFH} with the heat kernel estimates~\eqref{S2heatkernel} and the upper bounds~\eqref{etaU} of  $\eta_{t}^{\alpha}$ to get
		$$P_t^{\alpha, 0}(x)\,\varphi_{0}(x)^{-1}\lesssim\int_{0}^{+\infty}s^{-\frac{n}{2}-1-\frac{\alpha}{2}}\,
		\bigg\{
		\prod_{\alpha\in\Sigma_{r}^{+}}
		(1+s+\langle{\alpha,x^{+}}\rangle)^{\frac{m_{\alpha}+m_{2\alpha}}{2}-1}
		\bigg\}\e^{-\frac{|x|^2}{4s}} \diff s.$$
		Next, we split the right hand side
		into two integrals over $(0, (t^{1/\alpha}+|x|)^2]$ and $((t^{1/\alpha}+|x|)^2, +\infty)$. On the first interval we estimate $1+s+\langle{\alpha,x^{+}}\rangle \lesssim (t^{1/\alpha}+|x|)^2$ while on the second we estimate $1+s+\langle{\alpha,x^{+}}\rangle \lesssim s$.
		To conclude we proceed as in the proof of Proposition~\ref{Q kernel estimates}, taking into account that $|x|\asymp t^{1/\alpha}+|x|$.
		For the lower bound, it suffices to observe that by~\eqref{kernelFH} and~\eqref{etaUL} we have
		\[
		P_t^{\alpha,0}(x)\geq \int_{t^{2/\alpha}}^{+\infty}h_s^{0}(x)\,\eta_{t}^{\alpha}(s)\,\diff{s} \geq t \, \int_{(t^{1/\alpha}+|x|)^2}^{+\infty}s^{-1-\frac{\alpha}{2}}  h_s^{0}(x)\,\diff{s}.
		\]
		The claim follows by a similar argument to that of Proposition~\ref{Q kernel estimates} for $\zeta=0$. We omit the details.
	\end{proof}

	\section{Hardy--Littlewood maximal functions}  \label{Sec:maximal}
	Define, for $R>0$, the local maximal operator
	\[
	\mathcal{M}_{R}f = \sup_{0<r<R} \mathcal{A}_{r} f, \qquad \mathcal{A}_{r} f(xK)= \frac{1}{\mu(B(xK,r))} \int_{B(xK,r)} |f|\, \dd \mu, \qquad x\in G.
	\]
	Observe that by left invariance and~\eqref{volgrowth}, $\mu(B(xK,r)) \approx_{R} r^{n}$ for all $0<r<R$. We recall that for all $R>0$ the maximal operator $\mathcal{M}_{R}$ is bounded from $L^{1}$ to $L^{1,\infty}$ (i.e.\ it is of weak type $(1,1)$) and on $L^{p}$, $1<p\leq \infty$; see e.g.~\cite{S81}. By $L^{p,\infty}$ for $1\leq p<\infty$ we shall classically denote the space of (equivalence classes of) measurable functions such that
	\[
	\|f\|_{p,\infty} := \sup_{s>0} \, s\,  \mu(\{xK\in \mathbb{X} \colon |f(xK)|>s \}^{1/p}
	\]
	is finite. If one replaces $\mu$ by the weighted measure $u \, \dd \mu$ for a weight $u$, we shall write $L^{p,\infty}(u)$. In this section we prove the following.
	
	\begin{lemma}\label{lemmamaximal}
		Let $v$ be a weight, and suppose that $1\leq p<\infty$ and $R>0$. The following statements are equivalent.
		\begin{itemize}
			\item[(i)] There exists a weight $u$ such that  $\mathcal{M}_{R}$ is bounded from $L^{p}(v)$ to $L^{p}(u)$, if $p>1$, or from $L^{1}(v)$ to $L^{1,\infty}(u)$, if $p=1$.
			\item[(ii)] There exists a weight $u$ such that $\mathcal{M}_{R}$  is bounded from $L^{p}(v)$ to $L^{p,\infty}(u)$.
			\item[(iii)] There exists a weight $u$ such that $\mathcal{A}_{R} $ is bounded from $L^{p}(v)$ to $L^{p,\infty}(u)$.
			\item[(iv)] $v\in \mathcal{D}_{p}^{\mathrm{loc}}$.
		\end{itemize}
	\end{lemma}
	
	The proof is a straightforward adaption of the proof of~\cite[Lemma 3.4]{HTV}, see also the one of~\cite[Theorem 2.1]{GHSTV}, provided an $\ell^{q}$-valued weak-type $(1,1)$ inequality for $\mathcal{M}_{R}$ is at disposal. Despite $\mathcal{M}_{R}$ being a \emph{local} maximal operator, it seems not possible to follow verbatim the classical proof on $\R^{n}$, see e.g.~\cite[Ch.\ II]{St}, as this makes use of the Calder\'on--Zygmund decomposition of a function, which on the nondoubling space $\mathbb{X}$ is not available. Inspired by~\cite{AL}, we shall however reduce matters to a doubling space -- where the result is known by~\cite{GLY}.
	
	Given a sequence of measurable functions $f=(f_{j})$ on $\mathbb{X}$, define  $\mathcal{M}_{R} f = (\mathcal{M}_{R} f_{j})$; if
	\begin{equation}\label{vvnorms}
		|f|_{q} =\Big( \sum_{j}| f_{j}|^{q}\Big)^{1/q}, \qquad |\mathcal{M}_{R} f|_{q} =  \Big( \sum_{j}| \mathcal{M}_{R} f_{j}|^{q}\Big)^{1/q},
	\end{equation}
	then after few preliminaries we shall prove the following.
	\begin{proposition}\label{vvw11}
		Suppose $R>0$ and $q\in (1,\infty)$. Let $f=(f_{j})$ be a sequence of measurable functions. Then the following holds.
		\begin{enumerate}
			\item If $|f|_{q} \in L^{1}$, then $ |\mathcal{M}_{R} f|_{q} \in L^{1,\infty}$ and 
			\[
			\| |\mathcal{M}_{R} f|_{q} \|_{1,\infty} \lesssim_{R,q} \| |f|_{q}\|_{1}.
			\]
			\item If $p\in (1,\infty)$ and $|f|_{q}\in  L^{p}$, then $|\mathcal{M}_{R} f|_{q} \in L^{p} $ and
			\[
			\| |\mathcal{M}_{R} f|_{q} \|_{p} \lesssim_{R,q} \| |f|_{q}\|_{p}.
			\]
		\end{enumerate}
	\end{proposition}
	
	\subsection{Vector-valued weak-type $(1,1)$ for $\mathcal{M}_{R}$.}
	For $w\in G$, define the left shift operator $\tau_{w}$ by
	\[
	(\tau_{w} f)(xK) = f(w^{-1}xK), \qquad x\in G,
	\]
	for all measurable functions $f$. Then $\mathcal{M}_{R}$ is left invariant, namely for all $w\in G$
	\begin{equation}\label{MRTI}
		\mathcal{M}_{R}(\tau_{w}f) = \tau_{w}(\mathcal{M}_{R} f).
	\end{equation}
	Observe moreover that if $f$ is supported in a set $E \subseteq \bbX$, then $\mathcal{M}_{R} f$ is supported in $E_{R} =\{ xK\in \mathbb{X} \colon d(xK,E)< R\}$. These properties actually depend on $\mathcal{A}_{r}$ being a convolution operator.
	
	Suppose now that $\varrho>0$ and consider the space of homogeneous type $(B(o,\varrho),d,\mu)$. Since $\mu(B(o,\varrho)\cap B(xK,s)) \approx_{\varrho}s^n $ for all $xK\in B(o,\varrho)$ and $0<s<\varrho$ (see e.g.~\cite[p.~1317]{AL}), the local Hardy--Littlewood maximal function on $(B(o,\varrho),d,\mu)$ is comparable to $\mathcal{M}_{R}$. Therefore, by~\cite[Theorem~1.2]{GLY} for $q\in (1,\infty)$ we have the vector-valued weak-type $(1,1)$ inequality
	\begin{equation}\label{localVV}
		s\, \mu \bigg( \bigg\{ xK\in B(o,\varrho)\colon \bigg( \sum_{j}|\mathcal{M}_{R}g_{j}(x)|^{q}\bigg)^{1/q}> s \bigg\} \bigg) \lesssim_{\varrho} \bigg\| \bigg( \sum_{j}| g_{j}|^{q}\bigg)^{1/q} \bigg\|_{1}
	\end{equation}
	for all $s>0$ and sequence of functions $(g_{j})$ such that $\supp g_{j} \subseteq B(o,\varrho)$ for all $j$'s. A similar vector-valued $L^{p}$ inequality holds for $p\in (1,\infty)$.  
	
	We also recall that by~\cite[Lemma 5]{AL}, there exists a countable partition $\mathbb{X} = \bigcup _{k}A_{k}$ of Borel subsets and points $(x_{k})$ of $G$ such that $B(x_{k}K,1) \subseteq  A_{k} \subseteq B(x_{k}K,2)$ for all $k$, and such that for all $\nu > 0$ there exists $N=N(\nu)$ such that each ball $B(x_{k}K, \nu)$ intersects at most $N$ other balls $B(x_{h}K, \nu)$. This latter property will be referred to as the \emph{bounded overlap property}.
	
	Therefore, we can decompose each measurable function $f$ as
	\begin{equation}\label{Ankerdec}
		f = \sum_{k}f \mathbf{1}_{A_{k}} = \sum_{k} f^{k}, \qquad f^{k}= f \mathbf{1}_{A_{k}}.
	\end{equation}
	Observe that $f^{k}$ is supported in $A_{k} \subseteq x_{k}B(o,2)$, whence $\tau_{x_{k}^{-1}}f^{h}$ is supported in $x_{k}^{-1}x_{h}B(o,2)$ for $h,k\in \N$. In particular, $\tau_{x_{k}^{-1}}f^{k}$ is supported in $B(o,2)$.
	
	We are now ready to prove Proposition~\ref{vvw11}.
	
	\begin{proof}[Proof of Proposition~\ref{vvw11}]
		We shall prove only the statement (1), being (2) analogous and actually easier (notice however that (2) cannot be obtained by interpolation from (1): in the vector valued setting, the $L^{\infty}$ boundedness of $\mathcal{M}_{R}$ fails (see~\cite{St})).
		
		We consider the quantity
		\begin{align*}
			\mu \bigg( \bigg\{ xK \in \mathbb{X}\colon \bigg( \sum_{j}|\mathcal{M}_{R}f_{j}(x)|^{q}\bigg)^{1/q}>s \bigg\} \bigg) = \sum_{k}\mu(E_{k}),
		\end{align*}
		where
		\[
		E_{k} =  \bigg\{  xK\in A_{k}\colon \bigg( \sum_{j}|\mathcal{M}_{R}f_{j}(x)|^{q}\bigg)^{1/q}>s \bigg\}.
		\]
		For $xK\in A_{k}$ write $x= x_{k}w$ with $wK\in x_{k}^{-1}A_{k}$. Then by~\eqref{MRTI}
		\[
		\mathcal{M}_{R}f_{j}(x) = \mathcal{M}_{R}f_{j}(x_{k}w) = \tau_{x_{k}^{-1}}(\mathcal{M}_{R}f_{j})(w) = (\mathcal{M}_{R}(\tau_{x_{k}^{-1}}f_{j}))(w),
		\]
		whence
		\begin{align*}
			E_{k} 
			& =x_{k}  \cdot \bigg\{ wK\in x_{k}^{-1}A_{k} \colon \bigg( \sum_{j}| (\mathcal{M}_{R}(\tau_{x_{k}^{-1}}f_{j}))(w)|^{q}\bigg)^{1/q}>s \bigg\} =: x_{k} F_{k},
		\end{align*}
		and by left invariance of the measure $\mu(E_{k}) = \mu(F_{k})$. We now decompose each $f_{j}$ as $f_{j} = \sum_{h} f_{j}^{h}$, recall~\eqref{Ankerdec}. The sublinearity of $\mathcal{M}_{R}$ and the triangle inequality in $\ell^q$ imply
		\begin{align*}
			\mu(E_{k})  
			&\leq  \mu \bigg( \bigg\{ wK\in x_{k}^{-1}A_{k} \colon \bigg( \sum_{j} \bigg| \sum_{h}(\mathcal{M}_{R}(\tau_{x_{k}^{-1}}f_{j}^{h}))(w)\bigg|^{q}\bigg)^{1/q}>s \bigg\}\bigg)\\
			& \leq \mu \bigg( \bigg\{ wK\in x_{k}^{-1}A_{k} \colon \sum_{h} \bigg( \sum_{j}|(\mathcal{M}_{R}(\tau_{x_{k}^{-1}}f_{j}^{h}))(w)|^{q}\bigg)^{1/q}>s \bigg\}\bigg).
		\end{align*}
		Observe now that since $\tau_{x_{k}^{-1}}f_{j}^{h}$ is supported in $x_{k}^{-1}x_{h}B(o,2)$, each function $(\mathcal{M}_{R}(\tau_{x_{k}^{-1}}f_{j}^{h}))$ is supported in
		\[
		(x_{k}^{-1}x_{h}B(o,2))_{R} = x_{k}^{-1}x_{h} (B(o,2))_{R} = x_{k}^{-1}x_{h} B(o,2+R),
		\]
		whence $\mu(E_{k})$ is bounded by
		\[
		\mu \bigg( \bigg\{ wK\in x_{k}^{-1}\Big(A_{k} \cap \bigcup_{h} x_{h} B(o,2+R) \Big)\colon \sum_{h}\! \bigg( \sum_{j}|(\mathcal{M}_{R}(\tau_{x_{k}^{-1}}f_{j}^{h}))(w)|^{q}\bigg)^{1/q}>s \bigg\}\bigg).
		\]
		Observe now that since $A_{k} \subseteq x_{k}B(o,2) \subseteq x_{k}B(o,2+R)$, if $A_{k} \cap \bigcup_{h} x_{h} B(o,2+R)$ is not empty then also $x_{k}B(o,2+R) \cap \bigcup_{h} x_{h} B(o,2+R)$ is not empty. By the bounded overlap property, for each $k$ there are at most $N$ indices $h$ for which the intersection $x_{k}B(o,2+R) \cap x_{h} B(o,2+R)$ is not empty. Therefore, there are $m=m(k)$ elements $x_{h_{1}^{k}}, \dots , x_{h_{m}^{k}}$, with $1\leq m(k) \leq N$ for all $k$, such that $x_{k}B(o,2) \cap x_{h_{j}^{k}} B(o,2+R)$ is not empty. Therefore $\mu(E_{k})$ is bounded by
		\begin{align*}
			&\mu \bigg( \bigg\{ wK\in\bigcup_{a=1}^{m}x_{k}^{-1} \big(A_{k} \cap x_{h_{a}^{k}} B(o,2+R)\big) \colon \sum_{\ell=1}^{m} \bigg(\sum_{j}|(\mathcal{M}_{R}(\tau_{x_{k}^{-1}}f_{j}^{h_{\ell}^{k}}))(w)|^{q}\bigg)^{1/q}> s \bigg\}\bigg)\\
			& \leq  \sum_{a=1}^{m} \mu \bigg( \bigg\{ wK\in x_{k}^{-1}\big( A_{k} \cap x_{h_{a}^{k}} B(o,2+R)\big) \colon \sum_{\ell=1}^{m} \bigg( \sum_{j}|(\mathcal{M}_{R}(\tau_{x_{k}^{-1}}f_{j}^{h_{\ell}^{k}}))(w)|^{q}\bigg)^{1/q}\!\!\!> s \bigg\}\bigg)\\
			&\lesssim_{N}  \sum_{a,\ell=1}^{m}\mu \bigg( \bigg\{ wK\in x_{k}^{-1}\big(A_{k} \cap x_{h_{a}^{k}} B(o,2+R)\big)  \colon \bigg( \sum_{j}|(\mathcal{M}_{R}(\tau_{x_{k}^{-1}}f_{j}^{h_{\ell}^{k}}))(w)|^{q}\bigg)^{1/q}\!\!>\frac{s}{N} \bigg\}\bigg).
		\end{align*}
		Observe now that  $x_{k}^{-1}A_{k} \subseteq B(o,2)$, while for $a=1, \dots, m$ the balls $x_{k}^{-1}x_{h_{a}^{k}} B(o,2+R)$ are all balls of radius $2+R$ intersecting $B(o,2)$; therefore 
		\[
		x_{k}^{-1}A_{k} \cap x_{k}^{-1}x_{h_{a}^{k}} B(o,2+R) \subseteq x_{k}^{-1}A_{k} \cup x_{k}^{-1}x_{h_{a}^{k}} B(o,2+R)  \subseteq B(o,4+R).
		\]
		On the other hand, $\tau_{x_{k}^{-1}}f_{j}^{h_{\ell}^{k}}$ is supported in $x_{k}^{-1}x_{h_{\ell}^{k}} B(o,2+R)$, whence in $B(o,4+R)$. By~\eqref{localVV} applied to $\varrho=4+R$, left invariance and the bounded overlap property,
		\begin{align}\label{VV11ineq}
			s \sum_{k}\mu(E_{k})   
			& \lesssim \sum_{k} \sum_{\ell=1}^{m} \int_{\bbX}  \bigg( \sum_{j}|  \tau_{x_{k}^{-1}} f_{j}^{h_{\ell}^{k}}|^{q}\bigg)^{1/q} \, \dd \mu\\
			& = \sum_{k} \sum_{\ell=1}^{m} \int_{\bbX}  \bigg( \sum_{j}| f_{j} \mathbf{1}_{A_{h_{\ell}^{k}}}|^{q}\bigg)^{1/q} \, \dd \mu \nonumber \\
			&  \lesssim \sum_{k}\int_{x_k B(o,4+R)} \bigg( \sum_{j}| f_{j}|^{q}\bigg)^{1/q} \, \dd \mu \nonumber \\
			&  \lesssim  \int_{\bbX} \bigg( \sum_{j}| f_{j} |^{q}\bigg)^{1/q} \, \dd \mu \nonumber,
		\end{align}
		which concludes the proof.
		\end{proof}

	\section{A suitable class of kernels}\label{sec:class}
	
	\begin{definition}
		Suppose $\gamma>0$. We say that a family of measurable functions $(\psi_t)_{t>0}$ on $\mathbb{X}$ \textit{belongs to the class $\mathcal{P}_{\gamma}$}, and write $(\psi_t) \in \mathcal{P}_{\gamma}$,  if 
		\begin{itemize}
			\myitem{(P1)}\label{P1} For all $t>0$, $\psi_t$ is positive,  belongs to $L^p_{\text{loc}}$ for all $1\leq p<\infty$ and it is symmetric on $G$, i.e. $\psi_{t}(x^{-1})=\psi_{t}(x)$ for all $x\in G$ (in particular, $\psi_{t}$ is a positive bi-$K$-invariant function on $G$);
			
			\myitem{(P2)}\label{P2} $\int_G \psi_t(x)\, \varphi_0 (x) \, \dd x $ is finite for all $t>0$, the functions $m_{t}:=\mathcal{H}(\psi_t)$ are uniformly bounded on $\mathfrak{a}$ (i.e.\ there exists $C>0$ such that $|m_t(\lambda)|\leq C$ for all $\lambda \in \mathfrak{a}$ and $t>0$) and such that $\lim_{t\to 0^{+}}m_t=1$;
			
			\myitem{(P3)}\label{P3} for all $R,C>0$, all $0<t<R$ and all $x$ with $|x|\leq C t^{\frac{1}{\gamma}}$
			\begin{equation}\label{eqP3-1}
				\psi_t(x)\asymp_{R,C} t^{-\frac{n}{\gamma}},
			\end{equation}
			and there is a constant $d_{\gamma}>0$ such that for all positive $f\in L_{\text{loc}}^{1}$
			\begin{equation}\label{eqP3-2}
				\sup_{0<t<R}f* (\mathbf{1}_{B(eK, (d_{\gamma}R)^{1/\gamma})}\psi_{t})\lesssim \mathcal{M}_R f;
			\end{equation}
			
			\myitem{(P4)}\label{P4} there is a constant $c_{\gamma}>0$ such that for all $x_0,x,y\in G$ and $t>0$
			\begin{equation}\label{eqP4}
				\frac{\psi_t(y^{-1}x )}{\psi_{c_{\gamma}t}(y^{-1}x_0)}\leq C(x_0,x,t);
			\end{equation}
			\myitem{(P5)}\label{P5} for all $R,a>0$ and $a\leq t\leq R$ there is $C(R,a)>0$ such that for all $x\in G$
			\begin{equation}\label{eqP5-1}
				\psi_t(x)  \leq C(R,a)\, \psi_{R}(x),
			\end{equation}
			while there is $C(R)>0$  such that for all $x$ with $|x|\geq (d_{\gamma} R)^{1/\gamma}$ and all $0<t<R$  
			\begin{equation}\label{eqP5-2}
				\psi_t(x)  \leq C(R)\, \psi_{R}(x),
			\end{equation}
			where $d_{\gamma}$ is the constant in~\ref{P3}.
		\end{itemize}
	\end{definition}
	Let us stress that the conditions~\ref{P1}--\ref{P5} have no pretense of being optimal for the pointwise convergence problem. The class $\mathcal{P}_{\gamma}$ and the above conditions are a mere device which allows us not to go through several steps of the proof of Theorem~\ref{teo:main} below three times.  For example, the symmetry condition in~\ref{P1} for the kernels in $\mathcal{P}_{\gamma}$ might be dropped, and one could require just that they are bi-$K$-invariant. The results would still hold, with obvious modifications on the proofs. We leave the details to the interested reader.

	\begin{remark}  \label{rem:Lploc}
		Suppose $(\psi_t) \in \mathcal{P}_{\gamma}$ for some $\gamma>0$. Since  $\int_G \psi_t(x)\, \varphi_0 (x) \, \dd x $ is finite and $|\varphi_{\lambda}|\lesssim \varphi_{0}$ for all $\lambda\in\mathfrak{a}$, the spherical Fourier transform $m_{t}$ (recall~\eqref{SFtransform}) is well defined on $\mathfrak{a}$. 
	\end{remark}

	We now proceed to showing that the heat kernel, the fractional Poisson kernel and the fractional heat kernel are examples of such classes.
	
	\subsection{The heat kernel}  We shall show that $(h_t^{\zeta})$ belongs to the class $\mathcal{P}_{2}$. To do this, we shall make use of the following straightforward lemma.
	
	\begin{lemma}\label{heat eucl}
		Suppose $\zeta \geq 0$ and $0<t<T$. Then
		\begin{align*}h_t^{\zeta}(x)&\asymp_{T}  t^{-\frac{n}{2}} 
			\bigg\{ \prod_{\alpha\in\Sigma_{r}^{+}}
			(1+\langle \alpha, x^{+} \rangle)^{\frac{m_{\alpha}+m_{2\alpha}}{2}-1}\bigg\}\varphi_{0}(x)\,\e^{-\frac{|x|^{2}}{4t}}, \qquad x\in G.
		\end{align*}
	\end{lemma}
	\begin{proof}
		Since $t$ is bounded $\e^{-\zeta^2 t}\asymp_{T} 1$, while for all positive roots $\alpha$ and $x^{+}\in \overline{\mathfrak{a}^{+}}$ 
		\[
		1+t+\langle \alpha, x^{+} \rangle \asymp_{T} 1+\langle \alpha, x^{+} \rangle.
		\]
		Then, the claim follows immediately by the bounds~\eqref{S2heatkernel}.
	\end{proof}
	
	\begin{proposition}\label{prop:heatP2}
		Suppose $\zeta\in[0, \|\rho\|]$. Then $(h_t^{\zeta}) \in \mathcal{P}_{2}$.
	\end{proposition}	
	
	\begin{proof}
		First of all, properties in~\ref{P1} are well known and have already been mentioned, while those in~~\ref{P2} follow by the fact that $h_t^{\zeta}\in L^1(\mathbb{X})$ and by observing that $m_{t}(\lambda) =\e^{-t(\|\lambda\|^2+\zeta^2)}$.	
		
		Let us now show~\ref{P3}, whence let $C,R>0$ be arbitrary and pick $0<t<R$ and $x$ such that $|x|\leq C\sqrt{t}$. Then $\e^{-|x|^2/(4t)}\asymp  1$, and by combining~\eqref{phi_0 upper} with the elementary estimates
		\[
		1\leq 1+\langle\alpha, x^{+}\rangle \leq 1+\|\alpha\||x|,  \qquad 		(1+|x|)^M \e^{-\rho_{\min}|x|} \lesssim 1,	
		\]
		we get
		\[
		\prod_{\alpha\in\Sigma_{r}^{+}} (1+\langle{\alpha,x^{+}}\rangle)^{\frac{m_{\alpha}+m_{2\alpha}}{2}-1}\,\varphi_{0}(x)\asymp_{R,C} 1\]
		and~\eqref{eqP3-1} follows. 
		
		We now prove~\eqref{eqP3-2} for  $d_2=2n$. 	
		If $k_{0} \in \N\cup\{0\}$ is such that $2^{k_{0}}t \leq R<2^{k_0+1}t$, then
		\begin{align*}
			\mathbf{1}_{B(e, (2nR)^{1/2})}\,h_{t}^{\zeta}
			&\leq  \mathbf{1}_{B(e, (2nt)^{1/2})}\,h_{t}^{\zeta} + \sum_{k=0}^{k_{0}} \mathbf{1}_{B(e,(2n2^{k+1}t)^{1/2})\setminus B(e,(2n2^{k}t)^{1/2})}\, h_{t}^{\zeta}.
		\end{align*}
		Using the pointwise estimates for  $h_t^{\zeta}$ we get on the one hand 
		\begin{align*}
			\mathbf{1}_{B(e, (2nt)^{1/2})} \,h_{t}^{\zeta}
			& \lesssim  \mu(B(e, (2nt)^{1/2})^{-1}\, \mathbf{1}_{ B(e,(2nt)^{1/2})},
		\end{align*}
		while on the other hand 
		\begin{align*}
			\mathbf{1}_{B(e,(2n2^{k+1}t)^{1/2})\setminus B(e,(2n2^{k}t)^{1/2})}\, h_{t}^{\zeta} & \lesssim t^{-\frac{n}{2}}\,e^{-n2^{k-1}}\,\mathbf{1}_{B(e,(2n2^{k+1}t)^{1/2})}\\
			&\lesssim M_k \, \mu(B(e,(2n2^{k+1}t)^{1/2})^{-1}\mathbf{1}_{B(e,(2n2^{k+1}t)^{1/2})},
		\end{align*}
		where $M_k=2^{nk/2}\exp(-n2^{k-1})$. Therefore, summing over all $k\geq 0$, we get
		\begin{equation*}
			\sup_{0<t<R}(f* \mathbf{1}_{B(e, (2nR)^{1/2})}h_{t}^{\zeta}) \lesssim \mathcal{M}_{R}f.
		\end{equation*}

		We next prove~\ref{P4} for $c_2=4$. We claim that it suffices to show that for all $t>0$ and $x,x_0, y\in G$ one has
		\begin{equation}\label{quot heat}
			\frac{h_t^{\zeta}(y^{-1}x )}{h_{4t}^{\zeta}(y^{-1}x_0)}\leq C'(x_0,x)\,\e^{-\frac{|y^{-1}x|^2}{4t}+\frac{|y^{-1}x_0|^2}{16t}}
		\end{equation}
		for some $ C(x_0,x)>0$. Assuming for a moment that~\eqref{quot heat} is true, let us complete the proof of~\ref{P4}.  Pick $x_0\in G$. For $x\in G$ consider  $B_{x} = \{y \colon d(xK,yK) \leq d(xK,x_0K)\}$, and split $ h_{t}^{\zeta}(y^{-1} x \, )   = h_{t}^{\zeta}(y^{-1}x  )   \mathbf{1}_{B_{x}}(y) + h_{t}^{\zeta}(y^{-1}x )  \mathbf{1}_{B_{x}^{c}}(y)$. If $y\in B_{x}$, then $|y^{-1}x_0| \leq |y^{-1}x| + |x_{0}^{-1}x| \leq 2|x^{-1}x_0|$. Thus
		\[
		\e^{- \frac{|y^{-1}x|^2}{4t}} \leq 1 = \e^{\frac{|x^{-1}x_0|^2}{4t}} \e^{- \frac{|x^{-1}x_0|^2}{4t}}  \leq \e^{\frac{|x^{-1}x_0|^2}{4t}} \e^{- \frac{\left(\frac{|y^{-1}x_0|}{2} \right)^{2}}{4t}}  = \e^{\frac{|x^{-1}x_0|^2}{4t}} \e^{-\frac{|y^{-1}x_0|^{2}}{16t}}
		\]
		and  due to~\eqref{quot heat} one has 
		\begin{align*}
			h_{t}^{\zeta}(y^{-1} x) \mathbf{1}_{B_{x}}(y) \leq C(x_0,x) \, h_{4t}^{\zeta}(y^{-1}x_0).
		\end{align*}
		Analogously, if $y\notin B_{x}$, then $|y^{-1}x_0| \leq |y^{-1}x| + |x_{0}^{-1}x| < 2|y^{-1}x|$, whence
		\[
		\e^{- \frac{|y^{-1}x|^2}{4t}} \leq \e^{- \frac{|y^{-1}x_0|^{2}}{16 t}},
		\]
		and as above
		\begin{align*}
			h_{t}^{\zeta}(y^{-1} x) \mathbf{1}_{B_{x}^{c}}(y) \lesssim C(x_0,x)\, h_{4t}^{\zeta}(y^{-1}x_0).
		\end{align*}
		It remains to prove the claim~\eqref{quot heat}. In other words, we need to compare $h_t^{\zeta}(x^{-1}y)=h_t^{\zeta}(y^{-1}x)$ to $h_{4t}^{\zeta}(x_0^{-1}y)=h_{4t}^{\zeta}(y^{-1}x_0)$.		First of all, observe that by~\eqref{S2 flat}, for every positive root $\alpha$, we have 
		$$|\langle \alpha, (y^{-1}x)^{+}-(y^{-1}x_0)^{+} \rangle|\leq \|\alpha\|\|(y^{-1}x)^{+}-(y^{-1}x_0)^{+}\|\leq \|\alpha\|\,|x_{0}^{-1}x|.$$
		Therefore, 
		\begin{equation}\label{polynomial parts}
			\begin{split}
				\frac{1+t+\langle \alpha, (y^{-1}x)^{+}\rangle}{1+4t+\langle \alpha, (y^{-1}x_0)^{+} \rangle } 
				&\leq \frac{1+4t+\langle \alpha, (y^{-1}x_0)^{+}\rangle+\|\alpha\||x^{-1}x_0|}{1+4t+\langle \alpha, (y^{-1}x_0)^{+} \rangle}  \\
				& \leq 1+\|\alpha\||x^{-1}x_0|.
			\end{split}
		\end{equation}
		Next,  due to~\eqref{S2 Distance} we have 
		$ \|A(kx)\|\leq |kx|=|x| $, for all $x\in G$, $k\in K$, thus by the formula~\eqref{phi split} we obtain 
		\begin{equation}\label{phi compared}
			\frac{\varphi_{0}(y^{-1}x)}{\varphi_{0}(y^{-1}x_0)}\leq \frac{\e^{\|\rho\||x|}\varphi_0(y)}{\e^{-\|\rho\||x_0|}\varphi_0(y)}=\e^{\|\rho\|(|x|+|x_0|)}.
		\end{equation}
		Therefore, using~\eqref{S2heatkernel},~\eqref{polynomial parts} and~\eqref{phi compared} we get
		\begin{equation*}
			\frac{h_t^{\zeta}(y^{-1}x)}{h_{4t}^{\zeta}(y^{-1}x_0)}\lesssim \bigg\{ \prod_{\alpha\in\Sigma_{r}^{+}} (1+\|\alpha\||x^{-1}x_0|)^{\frac{m_{\alpha}+m_{2\alpha}}{2}+1}\bigg\}\, \e^{\|\rho\|(|x|+|x_0|)} \, \e^{-\frac{|y^{-1}x|^2}{4t}+\frac{|y^{-1}x_0|^2}{16t}}
		\end{equation*}
		which completes the proof of the claim~\eqref{quot heat}, whence that of~\ref{P4}.	 
		
		Finally, we show~\ref{P5}. For $a\leq t\leq R$ property~\eqref{eqP5-1} follows immediately from Lemma~\ref{heat eucl}. To prove~\eqref{eqP5-2} for $d_2=2n$, assume that $0<t<R$ and $|x|\leq \sqrt{2nR}.$ Then, since $ t \mapsto t^{-\frac{n}{2}} \e^{- \frac{|x|^2}{4t}} \mathbf{1}_{B(e, \sqrt{2nR})^{c}}(x)$  is increasing  in the interval $(0,R)$ (as one can see by computing its derivative), by Lemma~\ref{heat eucl}
		\begin{align*}
			\mathbf{1}_{B(e, \sqrt{2NR})^{c}}(x)\,h_{t}^{\zeta} (x)
			&\approx  \Big\lbrace{
				\prod_{\alpha\in\Sigma_{r}^{+}}
				(1+\langle \alpha, x^{+} \rangle)^{\frac{m_{\alpha}+m_{2\alpha}}{2}-1}\Big\rbrace}\,\varphi_{0}(x)  \mathbf{1}_{B(e, \sqrt{2NR})^{c}}(x) \,t^{ -\frac{n}{2}} \e^{-\frac{|x|^2}{4t}} \\
			&  \lesssim  \Big\lbrace{
				\prod_{\alpha\in\Sigma_{r}^{+}}
				(1+\langle \alpha, x^{+} \rangle)^{\frac{m_{\alpha}+m_{2\alpha}}{2}-1}\Big\rbrace}\,\varphi_{0}(x) \mathbf{1}_{B(e, \sqrt{2nR})^{c}}(x)\, R^{ -\frac{n}{2}} \e^{-\frac{|x|^2}{4R}} \\
			&\lesssim  \mathbf{1}_{B(e, \sqrt{2	nR})^{c}}h_{R}^{\zeta} (x)\\
			& \lesssim  h_{R}^{\zeta}(x).
		\end{align*}
		The proof is complete.
	\end{proof}
	
	\subsection{The fractional Poisson kernel}\label{FPS}  As for the fractional Poisson kernel, we have the following.
	
	\begin{proposition}\label{propQP}
		Suppose $\zeta\in[0, \|\rho\|]$ and $\sigma \in (0,1)$. Then $(Q_{t}^{\sigma, \zeta})\in \mathcal{P}_1$.
	\end{proposition}
	\begin{proof}
		The properties in~\ref{P1} follow from the corresponding properties of the heat kernel and the subordination formula~\eqref{kernelCS}.  Those in~\ref{P2} follow by Proposition~\ref{Q kernel estimates} and the subordination formula~\eqref{kernelCS}, which allow to write (see also~\cite[(3.1)]{BP2022}) 
		\begin{align*}\mathcal{H}(Q_t^{\sigma, \zeta})(\lambda)=\frac{t^{2\sigma}}{4^{\sigma}\Gamma(\sigma)}\int_{0}^{+\infty}\e^{-s(\|\lambda\|^2+\zeta^2)} \e^{-\frac{t^2}{4s}}\frac{\diff s}{s^{1+\sigma}}\leq \frac{t^{2\sigma}}{4^{\sigma}\Gamma(\sigma)} \int_{0}^{+\infty}\e^{-\frac{t^2}{4s}}\frac{\diff s}{s^{1+\sigma}}=1
		\end{align*}
		by a change of variables and the definition of the Gamma function.	Finally, the claimed limit follows by dominated convergence. 
		
		The bounds~\eqref{eqP3-1} follow immediately from the Euclidean estimates of the kernel $Q_t^{\sigma, \zeta}$ of Proposition~\ref{Q kernel estimates}. As for~\eqref{eqP3-2}, we prove it for  $d_{1}=1$. 	
		If $k_{0} \in \N\cup\{0\}$ is such that $2^{k_{0}}t \leq R<2^{k_0+1}t$, then
		\begin{align*}
			\mathbf{1}_{B(e, R)}\,Q_{t}^{\sigma, \zeta}
			&\leq  \mathbf{1}_{B(e, t)}\,Q_{t}^{\sigma, \zeta} + \sum_{k=0}^{k_{0}} \mathbf{1}_{B(e,2^{k+1}t)\setminus B(e,2^{k}t)}\, Q_{t}^{\sigma, \zeta}.
		\end{align*}
		Using the Proposition~\ref{Q kernel estimates} we get on the one hand 
		\begin{align*}
			\mathbf{1}_{B(e, t)}  \,Q_{t}^{\sigma,\zeta}
			& \lesssim  \mathbf{1}_{ B(e,t)} \,  \mu(B(e,t))^{-1}
		\end{align*}
		while on the other hand 
		\begin{align*}\mathbf{1}_{B(e,2^{k+1}t)\setminus B(e,2^{k}t)}\, Q_{t}^{\sigma, \zeta}
			&\lesssim   t^{2\sigma} \,(t+2^{k}t)^{-(n+2\sigma)}\, \mathbf{1}_{B(e,2^{k+1}t)} \\
			&\lesssim M_k \, \mu(B(e,(2^{k+1}t))^{-1}\mathbf{1}_{B(e,2^{k+1}t)},
		\end{align*}
		where $M_k=2^{-2\sigma k}$. Therefore, summing over all $k\geq 0$, we get
		\begin{equation*}
			\sup_{0<t<R}(f* \mathbf{1}_{B(e, R)}\,Q_{t}^{\sigma,\zeta}) \lesssim \mathcal{M}_{R}f
		\end{equation*}
		which completes the proof of~\ref{P3}.
		
		To prove~\ref{P4}, we show that~\eqref{eqP4} holds with $c_1=1$. By Proposition~\ref{Q kernel estimates} the case $\zeta=0$ is analogous (actually easier), and omitted. 
		
		Suppose $t>0$ and pick $x, y, x_0\in G$. Let us start with the following trivial inequality: if $r>0$, then 
		\begin{align}
			\frac{t^r+|y^{-1}x_0|}{t^r+|y^{-1}x|}\leq \frac{t^r+|y^{-1}x|+|x_{0}^{-1}x|}{t^r+|y^{-1}x|} 
			\leq 1+\frac{|x^{-1}x_0|}{t^r}.\label{ineqP}
		\end{align}
		We next distinguish the following cases, being $t$ fixed, so that we can use Proposition~\ref{Q kernel estimates}.
		
		Suppose $|y^{-1}x| \leq 1$ and $|y^{-1}x_0| \leq 1$. Then by~\eqref{ineqP}
		$$\frac{Q_{t}^{\sigma, \zeta}(y^{-1}x)}{Q_{t}^{\sigma, \zeta}(y^{-1}x_0)}\asymp\left(\frac{t+|y^{-1}x_0|}{t+|y^{-1}x|}\right)^{n+2\sigma}\lesssim \left( 1+\frac{ |x^{-1}x_0|  }{t}\right)^{n+2\sigma}.$$

		Suppose now $|y^{-1}x| \geq 1$ and $|y^{-1}x_0| \geq 1$. Then by Proposition~\ref{Q kernel estimates}
		\begin{align*}\frac{Q_{t}^{\sigma, \zeta}(y^{-1}x)}{Q_{t}^{\sigma, \zeta}(y^{-1}x_0)}&\asymp\left(\frac{t+|y^{-1}x_0|}{t+|y^{-1}x|}\right)^{\frac{\ell}{2}+\frac{1}{2}+\sigma+|\Sigma_r^+|}\frac{\varphi_{0}(y^{-1}x)}{\varphi_{0}(y^{-1}x_0)}
			\\&\times \exp\left\{-\zeta(|y^{-1}x|-|y^{-1}x_0|)\frac{|y^{-1}x|+|y^{-1}x_0|}{\sqrt{t^2+|y^{-1}x|^2}+\sqrt{t^2+|y^{-1}x_0|^2}}\right\}.
		\end{align*}
		The claim follows using~\eqref{ineqP},~\eqref{phi compared} and the facts that 
		\begin{equation}\label{facts}
			-\zeta(|y^{-1}x|-|y^{-1}x_0|)\leq \zeta |x_{0}^{-1}x|,\qquad \frac{|y^{-1}x|+|y^{-1}x_0|}{\sqrt{t^2+|y^{-1}x|^2}+\sqrt{t^2+|y^{-1}x_0|^2}}\leq 1.
		\end{equation}

		Suppose $|y^{-1}x| \geq 1$ while $|y^{-1}x_0| \leq 1$. Then 
		\begin{align*}
			\frac{Q_{t}^{\sigma, \zeta}(y^{-1}x)}{Q_{t}^{\sigma, \zeta}(y^{-1}x_0)}&\asymp\frac{(t+|y^{-1}x_0|)^{n+2\sigma}}{(t+|y^{-1}x|)^{\frac{\ell}{2}+\frac{1}{2}+\sigma+|\Sigma_r^+|}}\,\varphi_0(y^{-1}x)\,\e^{-\zeta \sqrt{t^2+|y^{-1}x|^2}}\\
			&\lesssim   (t+1)^{n+\sigma -\frac{\ell}{2}-\frac{1}{2}-|\Sigma_r^+|} ,
		\end{align*}
		where we used that $\varphi_0(y^{-1}x) \lesssim 1$ by~\eqref{phi_0 upper} , and
		\[
		t+|y^{-1}x|\geq  t+1, \qquad  t+|y^{-1}x_0|\leq t+1.
		\]

		To conclude, suppose that $|y^{-1}x| \leq 1$ and $|y^{-1}x_0| \geq 1$. Then
		\begin{align*}
			\frac{Q_{t}^{\sigma, \zeta}(y^{-1}x)}{Q_{t}^{\sigma, \zeta}(y^{-1}x_0)}&\asymp\frac{(t+|y^{-1}x_0|)^{\frac{\ell}{2}+\frac{1}{2}+\sigma+|\Sigma_r^+|}}{(t+|y^{-1}x|)^{n+2\sigma}}\, \varphi_{0}(y^{-1}x_0)^{-1}\,\e^{\zeta\sqrt{t^2+|y^{-1}x_0|^2}}.
		\end{align*}
		By the triangle inequality $|y^{-1}x_0| \leq |x^{-1}x_0|+1$, whence
		\begin{align*}	
			\frac{(t+|y^{-1}x_0|)^{\frac{\ell}{2}+\frac{1}{2}+\sigma+|\Sigma_r^+|}}{(t+|y^{-1}x|)^{n+2\sigma}}
			\leq t^{-n-2\sigma} (t+|x^{-1}x_0|+1)^{\frac{\ell}{2}+\frac{1}{2}+\sigma+|\Sigma_r^+|}
		\end{align*}
		and
		\[
		\e^{\zeta \sqrt{t^2+|y^{-1}x_0|^2}} \leq \e^{{\zeta}\sqrt{t^2+(|x^{-1}x_0|+1)^2}}.
		\]
		Since now
		\[
		\varphi_{0}(y^{-1}x_0)^{-1}\lesssim \e^{\|\rho\||y^{-1}x_0|}\lesssim  \e^{\|\rho\||x^{-1}x_0|}
		\]
		by~\eqref{phi_0 upper}, the proof of~\ref{P4} is complete.
		
		Finally, we prove~\ref{P5}. If $a \leq t \leq R$, then~\eqref{eqP5-1}  follows easily, by considering
		again the cases $|x|$ small or large separately. 	On the other hand, by Proposition~\ref{Q kernel estimates}, for all  $t\in (0,R)$ and for all $|x|\geq R$, we have
		\begin{align*}
			\frac{Q_{t}^{\sigma, \zeta}(x)}{Q_{R}^{\sigma, \zeta}(x)}&\asymp \frac{t^{2\sigma}}{R^{2\sigma}} \left(\frac{R+|x|}{t+|x|}\right)^{\frac{\ell}{2}+\frac{1}{2}+\sigma+|\Sigma_r^+|}\e^{\zeta \frac{(R-t)(R+t)}{\sqrt{t^2+|x|^2}+\sqrt{R^2+|x|^2}}} \\
			&\lesssim \left(\frac{R}{|x|}+1\right)^{\frac{\ell}{2}+\frac{1}{2}+\sigma+|\Sigma_r^+|}\e^{\zeta \frac{2R^2}{R}}\\
			&\lesssim 1,
		\end{align*}
		which proves~\eqref{eqP5-1}.
	\end{proof}
	
	\subsection{The fractional heat kernel}
	
	\begin{proposition}\label{propPP}
		Suppose $\zeta\in[0, \|\rho\|]$ and $\alpha\in (0,2)$. Then $(P_t^{\alpha, \zeta})\in \mathcal{P}_{\alpha}$.
	\end{proposition}
	\begin{proof}
		Properties in~\ref{P1} follow from the corresponding properties of the heat kernel  and the subordination formula~\eqref{kernelFH}; for~\ref{P2}, it suffices to recall the well-known fact that  $\mathcal{H}({P_t^{\alpha, \zeta}})(\lambda) =\e^{-t(\|\lambda\|^2+\zeta^2)^{\alpha/2}}$.
		
		As for~\ref{P3},~\eqref{eqP3-1} with $\gamma=\alpha$ follow immediately from the euclidean estimates of the kernel $P_t^{\alpha, \zeta}$ of Proposition~\ref{P kernel estimates}.
		
		We now prove~\eqref{eqP3-2} for  $d_{\alpha}=1$. 	
		If $k_{0} \in \N\cup\{0\}$ is such that $2^{k_{0}}t \leq R<2^{k_0+1}t$, then
		\begin{align*}
			\mathbf{1}_{B(e, R^{1/\alpha})}\,P_{t}^{\alpha, \zeta}
			&\leq  \mathbf{1}_{B(e, t^{1/\alpha})}\,P_{t}^{\alpha, \zeta} + \sum_{k=0}^{k_{0}} \mathbf{1}_{B(e,(2^{k+1}t)^{1/\alpha})\setminus B(e,(2^{k}t)^{1/\alpha})}\, P_{t}^{\alpha, \zeta}.
		\end{align*}
		Using the pointwise estimates for $P_t^{\alpha, \zeta}$ of Proposition~\ref{P kernel estimates}, we get on the one hand 
		\begin{align*}
			\mathbf{1}_{B(e, t^{1/\alpha})} \,P_t^{\alpha, \zeta}
			& \lesssim \mu(B(e,t^{1/\alpha}))^{-1} \,  \mathbf{1}_{ B(e,t^{1/\alpha})}
		\end{align*}
		while on the other hand 
		\begin{align*}\mathbf{1}_{B(e,(2^{k+1}t)^{1/\alpha})\setminus B(e,(2^{k}t)^{1/\alpha})}\, P_t^{\alpha, \zeta}
			&\lesssim  \,  t \,(t^{1/\alpha}+2^{k/\alpha}t^{1/\alpha})^{-(n+\alpha)} \mathbf{1}_{B(e,(2^{k+1}t)^{1/\alpha})} \\
			&\lesssim M_k \, \mu(B(e,(2^{k+1}t)^{1/\alpha}))^{-1}\,\mathbf{1}_{B(e,(2^{k+1}t)^{1/\alpha})},
		\end{align*}
		where $M_k=2^{- k}$. Therefore, summing over all $k\geq 0$, we get
		\begin{equation*}
			\sup_{0<t<R}(f* \mathbf{1}_{B(e, R^{1/\alpha})}\,P_{t}^{\alpha,\zeta}) \lesssim \mathcal{M}_{R}f.
		\end{equation*}

		We now show property~\ref{P4} by showing~\eqref{eqP4} with $c_{\alpha}=1$. The case $\zeta=0$ is again similar and simpler by Proposition~\ref{P kernel estimates}, whence omitted. 
		
		Fix $t>0$ and pick $x, y, x_0\in G$. We distinguish the following cases, arising from Proposition~\ref{P kernel estimates}.
		
		Suppose \textit{$|y^{-1}x|<t^{2/\alpha}$, $|y^{-1}x_0|<t^{2/\alpha}$.} Then, by Proposition 
		\ref{P kernel estimates} and~\eqref{ineqP},
		$$\frac{P_{t}^{\alpha, \zeta}(y^{-1}x)}{P_{t}^{\alpha, \zeta}(y^{-1}x_0)}\asymp\left(\frac{t^{1/\alpha}+|y^{-1}x_0|}{t^{1/\alpha}+ |y^{-1}x|}\right)^{n+\alpha}\lesssim \left( 1+\frac{d(xK,x_0K)}{t^{1/\alpha}}\right)^{n+\alpha}.$$
		
		Suppose  \textit{$ |y^{-1}x|\geq t^{2/\alpha}$, $|y^{-1}x_0|\geq t^{2/\alpha}$.} Then, by  Proposition 
		\ref{P kernel estimates} we have
		\begin{align*}\frac{P_{t}^{\alpha, \zeta}(y^{-1}x)}{P_{t}^{\alpha, \zeta}(y^{-1}x_0)}&\asymp\left(\frac{t+|y^{-1}x_0|}{t+|y^{-1}x|}\right)^{\frac{\ell}{2}+\frac{1}{2}+\frac{\alpha}{2}+|\Sigma_r^+|}\frac{\varphi_{0}(y^{-1}x)}{\varphi_{0}(y^{-1}x_0)}\e^{-\zeta(|y^{-1}x|-|y^{-1}x_0|)}
		\end{align*}
		Then the claim follows using~\eqref{ineqP},~\eqref{phi compared} and~\eqref{facts}.		
		
		Suppose  \textit{$ |y^{-1}x|\geq t^{2/\alpha}$, $|y^{-1}x_0|< t^{2/\alpha}$.} Then by Proposition 
		\ref{P kernel estimates} we get
		\begin{align*}
			\frac{P_{t}^{\alpha, \zeta}(y^{-1}x)}{P_{t}^{\alpha, \zeta}(y^{-1}x_0)}&\asymp\frac{(t^{1/\alpha}+|y^{-1}x_0|)^{n+\alpha}}{(t+|y^{-1}x|)^{\frac{\ell}{2}+\frac{1}{2}+\frac{\alpha}{2}+|\Sigma_r^+|}}\,\varphi_0(y^{-1}x)\,\e^{-\zeta|y^{-1}x|}\\
			&\lesssim t^{-\frac{\ell}{2}-\frac{1}{2}-\frac{\alpha}{2}-|\Sigma_r^+|}(t^{1/\alpha}+t^{2/\alpha})^{n+\alpha},
		\end{align*}
		where we used that $\varphi_0(y^{-1}x) \lesssim 1$ by~\eqref{phi_0 upper}.
		
		Suppose $ |y^{-1}x|<t^{2/\alpha}$, $|y^{-1}x_0|\geq t^{2/\alpha}$.  Then
		\begin{align*}
			\frac{P_{t}^{\alpha, \zeta}(y^{-1}x)}{P_{t}^{\alpha, \zeta}(y^{-1}x_0)}&\asymp\frac{(t+|y^{-1}x_0|)^{\frac{\ell}{2}+\frac{1}{2}+\frac{\alpha}{2}+|\Sigma_r^+|}}{(t^{1/\alpha}+|y^{-1}x|)^{n+\alpha}}\, \varphi_{0}(y^{-1}x_0)^{-1}\,\e^{\zeta|y^{-1}x_0|}.
		\end{align*}
		Thus,  by the triangle inequality we get (write $\beta=\frac{\ell}{2}+\frac{1}{2}+\frac{\alpha}{2}+|\Sigma_r^+| $)
		\begin{align*}	
			\frac{(t+|y^{-1}x_0|)^\beta}{(t^{1/\alpha}+|y^{-1}x|)^{n+\alpha}}
			&=\frac{(t+|y^{-1}x_0|)^{\beta}}{(t+|y^{-1}x|)^\beta}\,
			\frac{(t+|y^{-1}x|)^{\beta}}{(t^{1/\alpha}+|y^{-1}x|)^{n+\alpha}} \\
			&\lesssim \left(1+\frac{|x^{-1}x_0|}{t}\right)^{\beta} t^{-\frac{n}{\alpha}-1} (t+t^{2/\alpha})^{\beta}.
		\end{align*}
		Next, using that
		\[
		|y^{-1}x_0|\leq |x_{0}^{-1}x|+ |y^{-1}x|\leq  |x^{-1}x_0|+t^{2/\alpha},
		\]
		using the lower bound in~\eqref{phi_0 upper} we get
		\[
		\varphi_{0}(y^{-1}x_0)^{-1}\lesssim \e^{\|\rho\||y^{-1}x_0|}\lesssim \e^{\|\rho\|( |x^{-1}x_0|+t^{2/\alpha})}, \qquad \e^{\zeta |y^{-1}x_0|} \leq \e^{\zeta( |x^{-1}x_0|+t^{2/\alpha})},
		\]
		which altogether prove~\eqref{eqP4} and conclude the proof of~\ref{P4}. 
		
		Finally, we prove~\ref{P5} by using repeatedly Proposition~\ref{P kernel estimates}. Suppose that $0<t<R$ and $|x|\geq R^{1/\alpha}$. If also $|x|\leq t^{2/\alpha}$, then $R^{1/\alpha} \leq |x|\leq R^{2/\alpha}$ (provided this interval is not empty) and $t+|x|$ is bounded; therefore
		\[
		\frac{P_t^{\alpha,\zeta}(x)}{P_R^{\alpha,\zeta}(x)}\asymp \frac{t}{R} \frac{(R^{1/\alpha}+|x|)^{n+\alpha}}{(t^{1/\alpha}+|x|)^{n+\alpha}}\lesssim 1,
		\]
		while if $|x|\geq t^{2/\alpha}$, in the case when $R^{1/\alpha} \leq |x|\leq R^{2/\alpha}$  we argue as before, while if $|x|\geq R^{2/\alpha}$ then 
		\[
		\frac{P_t^{\alpha,\zeta}(x)}{P_R^{\alpha,\zeta}(x)}\asymp \frac{t}{R} \frac{(R+|x|)^{\frac{\ell}{2}+\frac{1}{2}+\frac{\alpha}{2}+|\Sigma_r^+|}}{(t+|x|)^{\frac{\ell}{2}+\frac{1}{2}+\frac{\alpha}{2}+|\Sigma_r^+|}}\lesssim 1,
		\]
		which prove~\eqref{eqP5-2} with $d_{\alpha}=1$. Finally, it remains to consider
		$|x|\leq R^{1/\alpha}$ and $a<t<R$ to show~\eqref{eqP5-1}.  Then 
		\[
		\frac{P_t^{\alpha,\zeta}(x)}{P_R^{\alpha,\zeta}(x)}\asymp \frac{t}{R} \frac{(R^{1/\alpha}+|x|)^{n+\alpha}}{(t^{1/\alpha}+|x|)^{n+\alpha}}\lesssim \left(\frac{R^{1/\alpha}}{a^{1/\alpha}}\right)^{n+\alpha},
		\]
		and this finishes the proof.
	\end{proof}
	
	\subsection{Classes of weights} We conclude this section by introducing some classes of weights and proving some related results. We recall that a \emph{weight} is a strictly positive and locally integrable function on $\mathbb{X}$. 
	
	\begin{definition}
		Suppose $(\psi_t) \in \mathcal{P}_{\gamma}$ for some $\gamma>0$ and $1\leq p<\infty$, and let $v$ be a weight. We say that $v$ belongs to the class $\mathcal{D}_{p}(\psi_{t})$, and write $v\in \mathcal{D}_{p} (\psi_{t})$, if there exists $t_0> 0$ such that $\psi_{t_{0}}v^{-\frac{1}{p}} \in L^{p'}$, and that $v$ belongs to the class $\mathcal{D}_{p}^{\mathrm{loc}}$ ($v \in \mathcal{D}_{p}^{ \mathrm{loc}}$) if $v^{-\frac{1}{p}} \in L^{p'}_{\mathrm{loc}}$.
	\end{definition}

	\begin{lemma}\label{lemmaequiv}
		Suppose $(\psi_t) \in \mathcal{P}_{\gamma}$ for some $\gamma>0$ and $1\leq p<\infty$, and let $v$ be a weight.  Then, $v\in \mathcal{D}_{p}(\psi_{t})$ if and only if there exists $t_1>0$ such that $\psi_{t_{1}}(x^{-1} \cdot \, ) v^{-\frac{1}{p}}  \in L^{p'}$ for all $x\in G$.
	\end{lemma}
	
	\begin{proof}
		The ``if'' implication is obvious by taking $t_{0}=t_{1}$ and $x=e$.  As for the other implication, recall that by property~\ref{P4} of $(\psi_{t})$ with $t=t_{1} = \frac{t_{0}}{c_{\gamma}}$ and $x_{0}=e$, we have 
		\[
		\psi_{t_{1}}(x^{-1} \cdot \, )v^{-\frac{1}{p}} \leq C(x, t_0)\, \psi_{t_{0}}v^{-\frac{1}{p}},
		\]
		and this completes the proof.
	\end{proof}

	\begin{corollary}\label{Cor Dploc}
		Suppose $1\leq p<\infty$, $\gamma>0$ and $(\psi_t) \in \mathcal{P}_{\gamma}$. Then $\mathcal{D}_{p}(\psi_{t})\subseteq \mathcal{D}^{p}_{\mathrm{loc}}$.
	\end{corollary}
	
	\begin{proof}
		Suppose $v\in \mathcal{D}_{p}(\psi_{t})$. It is enough to prove that for all $r>0$, the function $\mathbf{1}_{B(o,r)} v^{-\frac{1}{p}}$ belongs to $L^{p'}$. Observe that to show this, if $t_{1}>0$ is that of Lemma~\ref{lemmaequiv}, it is in turn enough to prove that $\mathbf{1}_{B(xK,t_1^{1/\gamma})} v^{-\frac{1}{p}} \in L^{p'}$ for all $x\in G$. Indeed, by compactness, every ball $B(o,r)$ can be covered by a finite number of balls $B(xK,t_1^{1/\gamma})$.
		
		But this follows since, given $x,y\in G$,
		\[
		\mathbf{1}_{B(xK,t_1^{1/\gamma})}(y) v^{-\frac{1}{p}} (y) = \mathbf{1}_{B(xK,t_1^{1/\gamma})}(y) \, t_1^{n/\gamma}\,t_1^{-n/\gamma} v^{-\frac{1}{p}} (y) \lesssim_{t_1} \psi_{t_{1}}(x^{-1}y) v^{-\frac{1}{p}} (y),
		\]
		by property~\ref{P3} and the fact that $\psi_{t_{1}}(x^{-1}\cdot ) v^{-\frac{1}{p}} \in L^{p'}$ by Lemma~\ref{lemmaequiv}.
	\end{proof}

	\section{Pointwise convergence}\label{sec:pc}
	
	In this section we consider convolution operators with right-convolution kernels from a class  $\mathcal{P}_{\gamma}$. More precisely, we consider operators
	\[
	T_tf(xK)=f\ast \psi_t(xK)=\int_{G}f(y)\psi_t(y^{-1}x)\,\diff y, \qquad x\in G, \; t>0,
	\]
	where $(\psi_t)$ is a family of functions in $\mathcal{P}_{\gamma}$ for some $\gamma>0$ which we fix all throughout. Recall that $\psi_t\in L^1_{\text{loc}}$ for all $t>0$, cf.~Remark~\ref{rem:Lploc}, whence $T_tf$ is well defined e.g.\ for $f\in C_{c}$. Our aim is to examine their behavior as $t\rightarrow 0^{+}$, for suitable classes of functions. We begin with a lemma.

	\begin{lemma}\label{lemmatest}
		If $f\in C_{c}^{\infty}$, then $T_t f(xK) \to_{t\to 0^{+}} f(xK)$ for all $x\in G$.
	\end{lemma}
	\begin{proof}
		Since $f\in C_{c}^{\infty}$, its Helgason--Fourier transform~\cite[Ch.~III, \S 1]{Hel1994},
		\[
		{\widehat{f}}(\lambda,kM)=\,\int_{G}\,f(xK)\,
		\e^{\langle -i\lambda+\rho, A(k^{-1}x)\rangle }\,\diff{x}, \quad \lambda \in \mathfrak{a}, \; k\in K,
		\]
		is well defined, $M$ being the centralizer of $\exp \mathfrak{a}$ in $K$.  Recall also that $f$ can be recovered by the inversion formula~\cite[Ch.~III, Theorem~1.3]{Hel1994}
		\begin{align}\label{H-Ftr-inv}
			f(xK)=\frac{1}{|W|}\int_{\mathfrak{a}}\int_K \e^{\langle i\lambda+\rho, A(k^{-1}x)\rangle } {\widehat{f}}(\lambda,kM)\,\diff k\,\frac{\diff\lambda}{|\textbf{c}(\lambda)|^2}.
		\end{align}
		By the Paley--Wiener theorem~\cite[Ch.~III, Theorem 5.1]{Hel1994}, for all $N\in \N$ 
		\begin{equation}\label{PW}
			| {\widehat{f}}(\lambda,kM)|\lesssim_{N}(1+\|\lambda\|)^{-N},  \qquad \lambda \in \mathfrak{a}, \; k\in K.
		\end{equation}
		Since now {$\psi_t$ is bi-$K$-invariant, we have}
		\[
		{\widehat{(f\ast\psi_t)}(\lambda, kM)} = {\widehat{f}}(\lambda, kM) \,  m_t(\lambda), 
		\]
		cf.~\cite[Ch.~III, Lemma 1.4]{Hel1994}, whence by~\eqref{H-Ftr-inv} and~\eqref{S2 Distance}
		\begin{align*}
			|(f\ast\psi_t-f)(xK)|\lesssim \e^{\|\rho\||x|}\int_{\mathfrak{a}}\int_{K}| \widehat{f} (\lambda,kM)|\,|1-m_t(\lambda)|\,\diff k\,\frac{\diff\lambda}{|\textbf{c}(\lambda)|^2}.
		\end{align*}
		It remains to observe that, by~\eqref{HCest},~\eqref{PW} and~\ref{P2}, for $N\in \N$
		\[
		|{\widehat{f}}(\lambda,kM)|\,| 1- m_t(\lambda)|\,|\textbf{c}(\lambda)|^{-2}\lesssim_{N} (1+\|\lambda\|)^{-N}(1+\|\lambda\|)^{m_{0}},
		\]
		uniformly in $t>0$, and that the right hand side is integrable on $K\times \mathfrak{a}$ if $N$ is sufficiently large. Then the conclusion follows by~\ref{P2} and dominated convergence.
	\end{proof}

	\begin{proposition}\label{propequiv}
		Suppose $1\leq p <\infty$ and let $v$ be a weight. The following statements are equivalent.
		\begin{itemize}
			\item[{(a)}] There exists $t_{0}>0$ and a weight $u$ such that $T_{t_0}$ is bounded from $L^p(v)$ to $L^{p}(u)$ if $p > 1$, or from $L^1(v)$ to $L^{1,\infty}(u)$ if $p=1$.
			\item[{(b)}]  There exists $t_{0}>0$ and a weight $u$ such that $T_{t_0}$ is bounded from $L^p(v)$ to $L^{p,\infty}(u)$.
			\item[{(c)}] There exists $t_{0}>0$ such that $T_{t_0}f$ is finite $\mu$-a.e.\ for all $f\in L^{p}(v)$.
			\item[{(d)}] $v\in \mathcal{D}_{p}(\psi_{t})$.
		\end{itemize}
	\end{proposition}
	
	\begin{proof}
		The implications $(a) \Rightarrow (b) \Rightarrow (c)$ are obvious. Let us now prove that $(c) \Rightarrow (d)$.
		
		We suppose $(c)$, and given a positive $f\in L^{p}(v)$ we pick $x_{0} \in G$ such that $T_{t_0}f(x_{0}K)$ is finite. By~\ref{P4}, $T_{\frac{t_{0}}{c_{\gamma}}}f(xK)$ is finite for all $x\in G$, whence the functional
		\[
		f \mapsto T_{\frac{t_{0}}{c_{\gamma}}}f(xK) = \int_{G}  \psi_{\frac{t_{0}}{c_{\gamma}}}(y^{-1}x) v^{-\frac{1}{p}}(y) f(y) v^{\frac{1}{p}}(y) \, \diff{y}
		\]
		is well defined for all $f\in  L^{p}(v)$ and $x\in G$. By duality, since $ fv^{1/p} \in L^{p}$, the function $ g(y)=\psi_{\frac{t_{0}}{c_{\gamma}}}(y^{-1} x) v^{-1/p}(y) =  \psi_{\frac{t_{0}}{c_{\gamma}}}(x^{-1}y)  v^{-1/p}(y) $ belongs to $L^{p'}$ for almost every $x\in G$, from which $(d)$ follows  by Lemma~\ref{lemmaequiv}.

		Assume now that $(d)$ holds. If $u$ is a weight and $t>0$, then by the fact that $\psi_{t}(z) = \psi_{t}(z^{-1})$ by~\ref{P1} and H\"older's inequality
		\begin{align*}
			|T_t f(xK)|  
			&\leq \int_{G}\psi_{t}(x^{-1}y) v^{-\frac{1}{p}}(y) |f(y)| v^{\frac{1}{p}}(y) \,\diff{y} \\
			& \leq  \| \psi_{t}(x^{-1}\cdot )v^{-\frac{1}{p}}\|_{p'} \|f v^{\frac{1}{p}}\|_{p} = \| \psi_{t}(x^{-1}\cdot )v^{-\frac{1}{p}}\|_{p'} \|f \|_{L^p(v)},
		\end{align*}
		whence
		\[
		\| T_t f\|_{L^{p}(u)}^{p}  \leq \|f \|_{L^p(v)}^{p} \int_{G} \| \psi_{t}(x^{-1}\cdot )v^{-\frac{1}{p}}\|_{p'}^{p}\, u(x)\, \diff{x}.
		\]
		By  Lemma~\ref{lemmaequiv}, there exists $t_{0}$ such that the function 
		\[
		x\mapsto \| \psi_{t_{0}}(x^{-1}\cdot )v^{-\frac{1}{p}}\|_{p'}^{p} =:\Psi(x)
		\]
		is finite for all $x\in G$. It is enough then to choose $u$ such that  $ \Psi u \in L^{1}$ and $(a)$ follows.
	\end{proof}
	We are now ready to prove the main result of the paper, which is an expanded version of Theorem~\ref{teointro}. For $R>0$, define the maximal operator
	\[
	\mathcal{T}^{*}_{R} f(xK) = \sup_{0<t<R} | T_tf(xK)|, \qquad x\in G.
	\]
	
	\begin{theorem}\label{teo:main}
		Suppose $v$ is a weight and $1\leq p <\infty$. The following statements are equivalent.
		\begin{enumerate}
			\item There exists $R>0$ and a weight $u$ such that $\mathcal{T}^{*}_{R}$ is bounded from $L^{p}(v)$ to $L^{p}(u)$ if $p > 1$, or from $L^{1}(v)$ to $L^{1,\infty}(u)$ if $p=1$.
			\item There exists $R>0$ and a weight $u$ such that $\mathcal{T}^{*}_{R}$ is bounded from $L^{p}(v)$ to $L^{p,\infty}(u)$.
			\item There exists $R>0$ such that $T_{R}f$ is finite $\mu$-a.e., and $\lim_{t\to 0^{+}} T_{t}f =f$  $\mu$-a.e.\ for all $f\in L^{p}(v)$.
			\item There exists $R>0$ such that $\mathcal{T}^{*}_{R}f$ is finite $\mu$-a.e.  for all $f\in L^{p}(v)$.
			\item $v\in \mathcal{D}_{p}(\psi_{t})$.
		\end{enumerate}
	\end{theorem}
	\begin{proof}
		The implication $(1) \Rightarrow (2)$ is obvious, while $(2) \Rightarrow (3)$ is implied by Lemma~\ref{lemmatest}, the density of $C_{c}^{\infty}$ in $L^{p}(v)$ and~\cite[Theorem 2.1.14]{G}.
		
		Suppose then that $(3)$ holds. We can assume without loss of generality that $f\geq 0$; let $x$ be such that $f\ast\psi_R(x)<\infty$ and   $\lim_{t\to 0^{+}} f\ast\psi_t(x)$ exists and is finite. Therefore, there exists $t_{x,f}=t(x,f)>0$ such that
		\begin{equation}\label{Cxftxf}
			\mathcal{T}^{*}_{t_{x,f}} f(x) = \sup_{0<t<t_{x,f}} f\ast\psi_t(x) <\infty.
		\end{equation}
		We can assume $t_{x_{0},f}<R$. By~\ref{P5} we get
		\[
		\mathcal{T}^{*}_{R}f(x) \leq \mathcal{T}^{*}_{t_{x,f}}f(x) + \sup_{t_{x,f}<t<R} f\ast\psi_t(x) \leq \mathcal{T}^{*}_{t_{x,f}}f(x)  + C(R,t_{x,f})f\ast\psi_R (x)
		\]
		which is finite by assumption and~\eqref{Cxftxf}, and $(4)$ follows.
		
		If $(4)$ holds, then Proposition~\ref{propequiv}~(c) holds, whence (d), i.e.\ (5).
		
		Finally, assume (5) and suppose $0<t<R$ for some $R>0$. We split 
		\[
		f*\psi_t = (f* \mathbf{1}_{B(e, (d_{\gamma}R)^{1/\gamma})}\psi_{t}) + (f* \mathbf{1}_{B(e, (d_{\gamma}R)^{1/\gamma})^c}\psi_{t}).
		\]
		By~\ref{P3} we have 
		\[
		\sup_{0<t<R}(f* \mathbf{1}_{B(e, (d_{\gamma}R)^{1/\gamma})}\psi_{t}) \lesssim \mathcal{M}_{R}f,
		\] 
		while if $t\in (0,R)$ and $|x|\geq (d_{\gamma}R)^{1/\gamma}$ then~\ref{P5} implies 
		\begin{align*}
			\sup_{0<t<R} (f* \mathbf{1}_{B(e, (d_{\gamma}R)^{1/\gamma})^{c}}\,\psi_{t}) 
			& \lesssim (f* \mathbf{1}_{B(e, (d_{\gamma}R)^{1/\gamma})^{c}}\psi_{R}) \lesssim (f*\psi_{R}).
		\end{align*}
		We conclude that
		\[
		\mathcal{T}^{*}_{R}f  \lesssim \mathcal{M}_{R}f + T_{R}f.
		\]
		Let now $R>0$ and $u_{1}$ be the weight given by  Proposition~\ref{propequiv}~(a), and $u_{2}$ be the weight given by Lemma~\ref{lemmamaximal}~(a). Then $(1)$ follows with $u= \min(u_{1},u_{2})$.
	\end{proof}

	\section{Pointwise convergence and the distinguished Laplacian}\label{sec:distinguished}
	
	In this final section we point out how the above results can be obtained also for the Cauchy problems~\eqref{Cauchyintro} associated to the so-called distinguished Laplacian of $\mathbb{X}$. We will not provide all the details, but rather only present the adaptations that the arguments need.

	Let $S=N(\exp{\mathfrak{a}})=(\exp{\mathfrak{a}})N$ be the solvable group appearing in the Iwasawa decomposition of $G$. As a manifold, $S$ can be identified with $\mathbb{X}$. The non-negative distinguished Laplacian $\tilde{\Delta}$ on $S$ is 
	given by
	\[
	\tilde{\Delta}=\tilde{\delta}^{1/2} \, \Ls_{0} \, \tilde{\delta}^{-1/2}
	\]
	where $\tilde{\delta}$ is the modular function of $S$, namely
	\begin{align*}
		\tilde{\delta}(g)=\tilde{\delta}(n(\exp{A}))=\e^{-2\langle{\rho,A}\rangle}, \qquad g\in S.
	\end{align*}
	Here $n=n(g)$ denotes the $N$-component of $g$ in the Iwasawa decomposition $G=N (\exp\mathfrak{a}) K$, while as before $A=A(g) \in \mathfrak{a}$ denotes the middle component.
	The distinguished Laplacian $\tilde{\Delta}$ is left-$S$-invariant and
	self-adjoint with respect to the right Haar measure $\varrho $ on $S$, which is given by
	\begin{align*}
		\int_{S}f(g)\, \textrm{d}\varrho(g) =\int_{N} \int_{\mathfrak{a}}f(n(\exp{A}))\, \dd A \, \dd n=\int_{\mathfrak{a}} \e^{2\langle{\rho,A}\rangle} \int_{N}f((\exp{A})n) \, \dd n\, \dd A.
	\end{align*}
	If  {$\dd \lambda (g) =  { \tilde{\delta}(g) } \, \dd \varrho (g)$} is the associated left Haar measure on $S$, one also has
	\begin{align}\label{drdldg}
		\int_{S}f(g)\,  { \textrm{d} \varrho(g) }=\int_{G} f(g) \, \e^{2\langle{\rho,A(g)}\rangle}\, \diff{g}
		\qquad\textnormal{and}\qquad
		\int_{S}f(g)\,  {\textrm{d} \lambda(g) } = \int_{G}f(g) \, \diff{g}.
	\end{align}
	Observe that $ {\varrho}(gE) = \tilde{\delta}(g)^{-1} {\varrho}(E)$ for all $g\in G$ and all measurable subsets $E \subseteq S$.

	\subsection{Cauchy problems}
	
	The Cauchy problem for the heat equation on $S$ associated with the distinguished Laplacian is
	\begin{align*}
		\begin{cases} \partial_{t}\tilde{u} + \tilde{\Delta}\tilde{u}=0, \qquad  \; t>0,\\
			\tilde{u}(\, \cdot \, , 0)=\tilde{f},
		\end{cases}	
	\end{align*}
	and the corresponding heat kernel is given by 
	\begin{equation}\label{heatdist}
		\tilde{h}_{t}=\tilde{\delta}^{\frac12}\, \e^{\|\rho\|^{2}t}h_{t}=\tilde{\delta}^{\frac12}h_{t}^{0},
	\end{equation}
	in the sense that
	\begin{align*}
		(\e^{-t\tilde{\Delta}\,} \tilde{f})(g)=(\tilde{f}*\tilde{h}_{t})(g)=\int_{S}\tilde{f}(y)\,\tilde{h}_{t}(y^{-1}g)\,\textrm{d}  {\lambda (y)} =\int_{S}\tilde{f}(gy^{-1})\,\tilde{h}_{t}(y)\, \dd  {\varrho (y)}.
	\end{align*}
	Here, we still denote by $*$ the convolution product on $S$ or on $G$. We refer to~\cite{Bou1983,CGGM1991} for more details.

	The Caffarelli--Silvestre extension problem for $\tilde{\Delta}$ on $S$ writes instead 
	\begin{align*}
		\begin{cases}
			\tilde{\Delta} \tilde{u} - \frac{(1-2\sigma)}{t} \partial_{t} \tilde{u} -  \partial^{2}_{tt} \tilde{u}=0, \qquad  \; t>0, \\
			\tilde{u}(\, \cdot \, ,0)=\tilde{f}
		\end{cases} 
	\end{align*}
	and this, by~\eqref{heatdist} and the subordination formula~\eqref{kernelCS}, still admits a fundamental solution $\tilde{Q}_t^{\sigma}=\tilde{\delta}^{1/2}\,Q_t^{\sigma, 0}$. The use of the latter is justified by the arguments in~\cite{BanEtAl} presented in Subsection~\ref{subsection CS} and observing that, if $f\in L^2(\mathbb{X})$ and $\tilde{f}=\tilde{\delta}^{1/2}f$, then by~\eqref{drdldg}
	\[
	\int_{S}\tilde{f}^2(g)\, \diff  {\varrho} (g) =\int_{S}f^2(g)\, \e^{-2\langle \rho, A(g)\rangle} {\diff \varrho} (g)=\int_{G} f^2(g) \,\diff{g}.
	\]
	Finally, for the fractional heat equation on $S$
	\begin{equation*}
		\begin{cases}
			\partial_t \tilde{u} +\tilde{\Delta}^{\alpha/2}\tilde{u} = 0,  \qquad \; t>0,\\
			\tilde{u}(\, \cdot \, ,0)=\tilde{f}
		\end{cases} 
	\end{equation*} 
	the solution is given by $\tilde{u}=\tilde{f}\ast \tilde{P}_t^{\alpha}$, where $\tilde{P}_t^{\alpha}=\tilde{\delta}^{1/2}\,P_t^{\alpha, 0}$ again in view of the subordination formula~\eqref{kernelFH}.
	
	The analogue of the class $\mathcal{P}_{\gamma}$ is now the following.
	
	\begin{definition}
		Suppose $\gamma>0$. We say that a family of measurable functions $(\tilde{\psi}_t)_{t>0}$ on $S$ \textit{belongs to the class $\tilde{\mathcal{P}}_{\gamma}$} if $\tilde{\psi}_t=\tilde{\delta}^{1/2}\, \psi_t$  with $(\psi_t)\in \mathcal{P}_{\gamma}$.
	\end{definition}
	Observe that this ensures that $\tilde{\psi}_t$ are also positive functions, satisfying the pointwise estimates~\ref{P3}--\ref{P5}. Indeed, the modular function $\tilde{\delta}$ is a continuous character of $S$. In particular, by the discussion above  and by Propositions~\ref{prop:heatP2},~\ref{propQP} and~\ref{propPP} the kernels $(\tilde{h}_t)$, $(\tilde{Q}_t^{\sigma})$ and $(\tilde{P}_t^{\alpha})$ belong to  $\tilde{\mathcal{P}}_{\gamma}$, for $\gamma=2$, $1$ and $\alpha$ respectively. As for the second assertion of~\ref{P3}, define, for $R>0$, the local maximal operator (with respect to the right measure) on $S$
	\[
	\tilde{\mathcal{M}}_{R}f = \sup_{0<r<R} \tilde{\mathcal{A}}_{r} f, \qquad \tilde{\mathcal{A}}_{r} f(x)= \frac{1}{\varrho(B(x,r))} \int_{B(x,r)} |f|\, \diff \varrho, \qquad x\in S,
	\]
	and observe that by~\eqref{drdldg} and since $\tilde \delta$ is bounded above and below away from zero on all compact sets, for all measurable functions $h$
	\begin{equation}\label{MRsxdx}
		\tilde{\mathcal{M}}_{R} h(x) \approx_{R}  \mathcal{M}_{R}h(x), \qquad x\in S.
	\end{equation}
	
	From now on, we shall denote by $L^p$, $1\leq p\leq \infty$, the Lebesgue spaces and by $L^{1,\infty}$ the Lorentz space with respect to the \emph{right} measure on $S$. If $v$ is a weight, i.e.\ a positive $L^{1}_{\mathrm{loc}}$ function on $S$, $L^p(v)$ (or $L^{p,\infty}(v)$) will be the spaces with respect to the measure  $v\, \dd{\varrho}$. The notations for the norms will remain unchanged. By $C_{c}$ and $C_{c}^{\infty}$ respectively we shall now mean the sets of continuous compactly supported and smooth and compactly supported functions on $S$.  The definition of the classes $\mathcal{D}_{p}$ needs to be modified accordingly.	
	
	\begin{definition}
		Suppose $(\tilde{\psi}_t) \in \tilde{\mathcal{P}}_{\gamma}$ for some $\gamma>0$ and $1\leq p<\infty$, and let $v$ be a weight. We say that $v$ belongs to the class $\tilde{\mathcal{D}}_{p}(\tilde \psi_{t})$ if there exists $t_0> 0$ such that $\tilde{\psi}_{t_{0}}v^{-\frac{1}{p}} \in L^{p'}$, and that $v$ belongs to the class $\tilde{\mathcal{D}}_{p}^{\mathrm{loc}}$ if $v^{-\frac{1}{p}} \in L^{p'}_{\mathrm{loc}}$.
	\end{definition}
	Then the analogues of Lemma~\ref{lemmaequiv} and Corollary~\ref{Cor Dploc} hold.
	
	Now consider convolution operators with right-convolution kernels from the class  $\tilde{\mathcal{P}}_{\gamma}$, $\gamma>0$, that is, operators
	\[
	\tilde{T}_t\tilde{f}(x)=\tilde{f}\ast \tilde{\psi}_t(x)=\int_{S}\tilde{f}(y)\,\tilde{\psi}_t(y^{-1}x)\,\diff  {\lambda (y)}, \qquad x\in S, \; t>0,
	\]
	for some $\gamma>0$ and $(\tilde{\psi}_t)\in \tilde{\mathcal{P}}_{\gamma}$. Observe that in view  of the symmetry of $\psi_t$, which implies
	$$\tilde{\psi}_t(z)=\tilde{\delta}(z)\tilde{\psi}_t(z^{-1}), \qquad z\in G,$$ the action of $\tilde{T}_t$ writes
	\begin{align}
		\tilde{T}_t \tilde{f}(x)
		&= \int_{S}\tilde{\psi}_{t}(y^{-1}x)  \tilde{f}(y) \, {\diff \lambda (y)}= \int_{S}\tilde{\psi}_{t}(y^{-1}x) \tilde{f}(y)\, \tilde{\delta}(y)\,  {\diff \varrho (y)} \notag \\
		&= \tilde{\delta}(x)\int_{S}\tilde{\psi}_{t}(x^{-1}y)\tilde{f}(y)  \,  {\diff \varrho (y)} \label{conv symm}.
	\end{align} 
	
	We are left with proving the analogues of Proposition~\ref{vvw11} and Lemma~\ref{lemmatest}, after which we will be essentially done. For a sequence of measurable functions $f=(f_{j})$ on $S$  define $ \tilde{\mathcal{M}}_{R}f = ( \tilde{\mathcal{M}}_{R}f_{j} )$ and the quantities $|f|_{q}$ and $|\tilde{\mathcal{M}}_{R} f|_{q}$ as in~\eqref{vvnorms}.

	\begin{proposition}\label{vvw11right}
		Suppose $R>0$ and $q\in (1,\infty)$. Let $f=(f_{j})$ be a sequence of measurable functions on $S$. Then the following holds.
		\begin{enumerate}
			\item If $|f|_{q} \in L^{1}$, then $ |\tilde{\mathcal{M}}_{R} f|_{q} \in L^{1,\infty}$ and 
			\[
			\| |\tilde{\mathcal{M}}_{R} f|_{q} \|_{1,\infty} \lesssim_{R,q} \| |f|_{q}\|_{1}.
			\]
			\item If $p\in (1,\infty)$ and $|f|_{q}\in  L^{p}$, then $\tilde{ \mathcal{M}}_{R} f \in L^{p} $ and
			\[
			\| |\tilde{\mathcal{M}}_{R} f|_{q} \|_{p} \lesssim_{R,q} \| |f|_{q}\|_{p}.
			\]
		\end{enumerate}
	\end{proposition}
	
	\begin{proof}
		Again we shall give the details of (1) only. By~\eqref{MRsxdx}, by maintaining the same notation as the one of the proof of Proposition~\ref{vvw11}, for all $s>0$ we have
		\begin{align*}
			s  \, \varrho \bigg( \bigg\{ x \in S\colon \bigg( \sum_{j}| \tilde{\mathcal{M}}_{R}f_{j}(x)|^{q}\bigg)^{1/q}> s \bigg\} \bigg)
			\approx s \, \sum_{k} \varrho(E_{k})  =  s \, \sum_{k}  \tilde{\delta}(x_{k})^{-1} \varrho(F_{k}).
		\end{align*}
		Since $F_{k}$ is contained in $B(e, 2)$, where $e$ is here the identity of $S$, one has ${\varrho}(F_{k}) \approx \mu(F_{k}) = \mu(E_{k})$, whence by~\eqref{VV11ineq}
		\begin{align*}
			s  \sum_{k}\varrho(E_{k})  
			&\lesssim \sum_{k}\tilde{\delta}(x_{k})^{-1} \sum_{\ell=1}^{m} \int_{S}  \bigg( \sum_{j}|  \tau_{x_{k}^{-1}} f_{j}^{h_{\ell}^{k}}|^{q}\bigg)^{1/q} \, \dd\varrho \\
			& \lesssim  \sum_{k}  \int_{B(x_{k},4+R)} \bigg( \sum_{j}| f_{j}|^{q}\bigg)^{1/q} \, \dd \varrho\\
			& \lesssim \int_{S} \bigg( \sum_{j}| f_{j} |^{q}\bigg)^{1/q} \, \dd\varrho,
		\end{align*}
		which concludes the proof.
	\end{proof}
	
	\begin{lemma}\label{lemmatestS}
		If $\tilde{f}\in C_{c}^{\infty}$, then $\tilde{T}_t \tilde{f}(x) \to_{t\to 0^{+}} \tilde{f}(x)$ for all $x\in S$.
	\end{lemma}
	\begin{proof} Define the function  {$f(xK)=\tilde{\delta}(x)^{-\frac{1}{2}}\tilde{f}(x)$ on $\mathbb{X}$} so that $f$ is a right-invariant smooth and compactly supported function on $G$. By~\eqref{drdldg}  then
		\begin{align*}
			\tilde{T}_t\tilde{f}(x)&=\tilde{f}\ast \tilde{\psi}_t(x)=\int_{S}\tilde{f}(y)\,\tilde{\psi}_t(y^{-1}x)\,\diff {\lambda}(y)=\tilde{\delta}^{\frac{1}{2}}(x)\int_G f(y)\, \psi_t(y^{-1}x)\, \diff{y}, 
		\end{align*}
		whence
		$$\tilde{T}_t\tilde{f}(x)-\tilde{f}(x)=\tilde{\delta}^{\frac{1}{2}}(x)(f\ast\psi_t(x)-f(x)).$$
		By Lemma~\ref{lemmatest}, the claim follows.
	\end{proof}
	It is now only a matter of using~\eqref{conv symm}, Proposition~\ref{vvw11right}, Lemma~\ref{lemmatestS} and the fact that the kernels $(\tilde{\psi}_t)$ satisfy~\ref{P3}--\ref{P5}, and repeating almost verbatim the arguments used in Section~\ref{sec:pc}, to get the analogue of Theorem~\ref{teo:main}. We omit the details.

\end{document}